\documentclass[12pt,a4paper]{article}
\usepackage{amsmath,amsthm,amssymb,amscd}
\usepackage[all]{xy}
\usepackage[dvips]{graphicx}

\newtheorem{thm}{Theorem}[section]
\theoremstyle{definition}

\theoremstyle{theorem}
\newtheorem{prop}[thm]{Proposition}

\newtheorem{lem}[thm]{Lemma}

\theoremstyle{remark}

\theoremstyle{definition}
\newtheorem{step}{Step}[]

\def\Ker{\mathop{\mathrm{Ker}}\nolimits}

\def\Hilb{{\mathrm{Hilb}}}
\def\<{{\langle}}
\def\>{{\rangle}}

\newcommand{\mf}[1]{{\mathfrak{#1}}}

\newcommand{\bb}[1]{{\mathbb{#1}}}
\newcommand{\mca}[1]{{\mathcal{#1}}}

\title{On the birational geometry for irreducible symplectic 4-folds related to the Fano schemes of lines}
\author{Kotaro Kawatani\footnote{kawatani@cr.math.sci.osaka-u.ac.jp}}
\date{Osaka University}

\begin{document}

\maketitle

\section{Introduction} 

A compact K\"{a}hler manifold $X$ is said to be an irreducible symplectic manifold 
when $X$ is a simply connected and $H^0(X, \Omega ^2_{X})$ is generated by 
an everywhere non-degenerate holomorphic 2-form $\sigma _X$. 
We call such a holomorphic 2-form $\sigma _X$ a symplectic form. 
Let $Y$ be a smooth cubic 4-fold and $F(Y)$ the set of all lines contained in $Y$. 
$F(Y)$ is said to be the Fano schemes of lines. 
Beauville and Donagi \cite{B-D} showed that $F(Y)$ is an irreducible symplectic 4-fold. 
Examples of irreducible symplectic manifolds are few. 
Most of them are constructed as the moduli spaces of stable sheaves on K3 surfaces or Abelian surfaces.

Assume that a finite group $G$ acts on $F(Y)$. 
If $G$ preserves a symplectic form $\sigma _{F(Y)}$ then there exists a symplectic form $\bar \sigma $ on the smooth locus of $X_{G}:=F(Y)/G$. 
In general $X_{G}$ has singular points. 
Let $\nu : \tilde X_{G} \to X_{G}$ be a resolution of $X_{G}$. 
The symplectic 2-form $\bar \sigma$ lifts to an everywhere non-degenerate 2-form on $\tilde X_{G}$
 if and only if the resolution $\nu$ is crepant. 
If $\nu$ is crepant, then $\tilde X_{G}$ is also an irreducible symplectic 4-fold. 
The $G$-action on $F(Y)$ is said to be a good action 
when $G$ preserves $\sigma _{F(Y)}$ and $X_{G}$ has a crepant resolution. 
Can we find good actions on $F(Y)$? 
When it is possible, what is $\tilde X_{G}$ ?

In this paper, 
we assume that $G$ is a subgroup of the automorphism group $\mathit{PGL}(5)$ of $\bb P^5$ and 
$G$ preserves $Y$. 
We study two examples of good actions on $F(Y)$. 
In the first example, the resolution $\tilde{X_{G}} $ is birational to 
the generalized Kummer variety $K^2(A) $ of an Abelian surface $A$. 
In the second one, $\tilde{X_{G}}$ is birational to the 2-points Hilbert scheme $\Hilb^2(S)$ of a K3 surface $S$. 
We have to investigate the $G$-action on symplectic forms on ${F(Y)}$. 
The Abel-Jacobi map defines a Hodge isomorphism of type $(-1,-1)$ from $H^4(Y, \bb C)$ to $H^2(F(Y),\bb C)$. 
Using the identification with $H^0(F(Y),\Omega ^2_{F(Y)})$ and $H^{1}(Y,\Omega_Y^3)$, 
we can check whether $G$ preserves the symplectic form or not (see Lemma \ref{criterion}).

In Section $3$, we consider the first example, which was found by Namikawa (\cite{Nam}, 1.7.(iv)). 
Let $Y$ be a smooth cubic 4-fold defined by $$\{ (z_0:\cdots :z_5) \in \bb P^5  | f(z_0,z_1,z_2) + g(z_3,z_4,z_5)=0 \},$$ 
where $f$ and $g$ are homogeneous cubic polynomials. 
In addition, 
we consider two elliptic curves defined by $ C := \{ (z_0:z_1:z_2) \in \bb P^2 | f(z_0,z_1,z_2)=0 \}$ and 
$ D := \{ (z_3:z_4:z_5) \in \bb P^2 | g(z_3,z_4,z_5)=0 \}$. 
Assume that $G$ is the cyclic group $\bb Z_3$ of order $3$. 
The group $G$ acts on $Y$ in the following way :
\[
\bb Z_3 {}^{\curvearrowright} Y\ni (z_0: \cdots :z_5) \mapsto (z_0:z_1:z_2:\zeta z_3: \zeta z_4:\zeta z_5)
\] 
where $\zeta = \exp(2\pi \sqrt{-1}/3)$. 
$X_{\bb Z_3}$ has singularities of type $A_2$ along $C \times D$ and is smooth otherwise. 
So $X_{\bb Z_3}$ has a crepant resolution $\tilde X_{\bb Z_3}$ (Lemma \ref{crepant}). 
$\tilde X_{\bb Z_3}$ is birational to $K^2(C \times D)$ for the elliptic curves $C$ and $D$ (Proposition \ref{mainprop}). 
There is a natural birational map $\psi:\tilde X_{\bb Z_3} \dashrightarrow  K^2 ( C \times D ) $ 
which is the Mukai flop on disjoint 18 copies of $\bb P^2$ (Theorem 3.5). 
Furthermore when $C$ and $D$ are not isogenus, we found 2 more symplectic birational models $X_{\mathrm{I}}$ and $X_{\mathrm{II}}$ other than $\tilde X_{\bb Z_3}$ and $K^2(C \times D)$ (Theorem \ref{mainthm2}).

In Section 4, we consider the second example. 
Let $Y$ be a smooth cubic 4-fold defined by $\{ (z_0: \cdots :z_5) \in \bb P^5 | z_0^3 + \cdots + z_5^3  =0\}$. 
Put  
\[
M =   \begin{Bmatrix}  A=\begin{pmatrix} a_1 &  &  \\  & \ddots & \\ & &  a_6\end{pmatrix} \in \mathit{PGL}(5)  |  \#\{ a_i=1 \} = \# \{j | a_j=\zeta  \} =3 \end{Bmatrix}.
\] 	
Let $G$ be the finite group generated by $M$. 
Then $G$ is isomorphic to $\bb Z_3^{\oplus 4}$ and the $G$-action on $F(Y)$ turns out to be good. 
Let $E$ be the elliptic curve defined by $\{ (z_0:z_1:z_2) \in \bb P^2 | z_0^3+z_1^3+z_2^3=0 \}$. 
Since $E$ has a complex multiplication $\zeta \curvearrowright E \ni x \mapsto \zeta x$, 
the Abelian surface $E \times E$ has the $\bb Z_3$ action given by $\bb Z_3 \curvearrowright E \times E \ni (x,y) \mapsto (\zeta x,\zeta ^2 y) \in E \times E$. 
This $\bb Z_3 $ action on $E \times E$ preserves a symplectic form on $E \times E$, 
and $(E \times E)/\bb Z_3$ has a crepant resolution $S \to (E \times E)/\bb Z_3$. In particular $S$ is a K3 surface. 
We shall prove that $\tilde{X_{G}}$ is birational to $\Hilb ^2(S) $ (Theorem 4.2). 
\vspace{1pt}

\noindent
\textbf{Acknowledgment}.
The author would like to thank Professor Namikawa for his advice. 

\section{Preliminaries}

In this section we prepare some facts needed later. 

\subsection{On irreducible symplectic manifolds}
Let $A$ be an Abelian surface and $\Hilb^{n+1}(A)$ the Hilbert scheme of $(n+1)$-points on $A$. 
We define 
\[
\mathrm{Sym}^{n+1}(A) := A^{n+1}/\mathfrak{S}_{n+1}
\]
where $\mathfrak{S}_{n+1}$ is the $(n+1)$-th symmetric group. 
Let $\mu$ be the Hilbert-Chow morphism:
\[
\mu: \Hilb^{n+1}(A) \to \mathrm{Sym}^{n+1}(A).
\] 
We define the map $\Sigma: \mathrm{Sym}^{n+1}(A) \to  A$ by $\Sigma  (\{ x_1, \cdots x_{n+1}\}):=\sum_{i=1}^{n+1} x_i$ and 
 consider the composite $\Sigma \circ \mu$.  
We define $:$
\[
K^n (A) := (\Sigma \circ \mu)^{-1}(0).
\] 
We call $K^n (A)$ \textit{the generalized Kummer variety of $A$}. 
$K^n(A)$ is an irreducible symplectic manifold (\cite{Bea}). 
We also define $\bar K ^n(A):=\Sigma ^{-1}(0)$. 
$\bar K ^n(A)$ is isomorphic to 
\[
\bar K ^n(A) \cong \{ (x_1, \cdots , x_{n+1}) \in A ^{n+1}| x_1 + \cdots + x_{n+1}=0   \}/ \mathfrak{S}_{n+1}.
\] 

For an elliptic curve $C$, we define  
\[
K^2(C) :=  \bigr\{ \{ p_1,p_2,p_3 \} \in \Hilb ^3(C) | p_1+p_2+p_3=0  \bigr\}.
\]
Notice that $K^2(C)$ is isomorphic to $\bb P^2$.

\subsection{Fano schemes}

We collect some facts for Fano schemes needed later. 
Let $Y$ be a smooth cubic 4-fold. Define
$$F(Y):=\{ l \subset Y | l \cong \bb P^1 , \deg l=1 \}. $$ 
$F(Y)$ is known to be a smooth projective variety of dimension 4. 
Furthermore we have the following theorem. 

\begin{prop}[Beauville-Donagi\cite{B-D}]
The notation being as above, $F(Y)$ is an irreducible symplectic manifold. 
$F(Y)$ is deformation equivalent to the Hilbert scheme $\Hilb ^2(S) $ of 2-point on a K3 surface $S$. 
\end{prop}

The following relation between $H^{4}(Y,\bb C)$ and $H^2(F(Y),\bb C)$ is important to us. 

\begin{lem}[\cite{B-D}]\label{Abel-Jacobi}
Let $\mca L $ be the universal family of lines$:$
$$\mca L := \{ (l,y) \in F(Y) \times Y | l \ni y \}.$$ 
Let $p$ and $q$ be the projections from $\mca L$ to $F(Y)$ and $Y$ respectively. 
We define the Abel-Jacobi map $\alpha $ by
\[
\alpha : H ^4(Y, \bb C) \to H^2(F(Y) , \bb C),\ \alpha (\omega  ):= p_*q^*(\omega ). 
\]
Then $\alpha $ is a Hodge isomorphism of type $(-1,-1)$. 
\end{lem}

Let $(z_0: \ldots : z_5) $ be the homogeneous coordinates of $\bb P^5$, 
and $f_Y(z_0, \ldots ,z_5)$ the defining polynomial of $Y$. 
Let $\mathit{Res}:H^5 (\bb P^5 -Y, \bb C )\to H^4(Y,\bb C) $ be the residue map. 
By Lemma \ref{Abel-Jacobi}, we have $ \alpha ( H^1 (Y , \Omega _{Y}^3)) = H^0(F(Y), \Omega_{F(Y)}^2)$.
Recall that	 $H^1 (Y , \Omega _Y^3)$ is generated by $Res \frac{\Omega }{f_Y^2}$ (\cite{Gri}), where 
$ \Omega= \sum _{i=0}^{5} (-1)^i z_i dz_0 \wedge   \cdots   \check{dz_i} \cdots \wedge dz_5 $. 
From these arguments, we get the following commutative diagram: 
\[
\xymatrix{
H ^4(Y, \bb C) \ar[r]_{\alpha}^{\cong } & H^2(F(Y),\bb C) \\
H^1 (Y , \Omega _{Y}^3) \ar@{^{(}->}[u] \ar[r]& H^0(F(Y), \Omega_{F(Y)}^2) \ar@{^{(}->}[u]\\
\bb C \langle \mathit{Res} \frac{\Omega}{f_Y^2} \rangle \ar@{=}[u] \ar[r] & \bb C \langle \sigma_{F(Y)} \rangle. \ar@{=}[u]
}
\]

\subsection{Group actions on $F(Y)$}

Let $\mathit{PGL}(5)$ be the automorphism group of $\bb P^5$, 
and $G$ a finite subgroup of $\mathit{PGL}(5)$. 
We assume that $G$ preserves $Y$. 
We need a criterion for the action preserving the symplectic form.

Let us consider the induced action $G {}^{\curvearrowright} F(Y)$. 
Since the Abel-Jacobi map $\alpha $ is $G$-equivariant, 
the induced action $G {}^{\curvearrowright} F(Y)$ preserves the symplectic form $\sigma _{F(Y)}$ if and only if 
$G$ preserves $\mathit{Res} \frac{\Omega }{f_Y^2}$.
So, we get the following lemma.

\begin{lem}\label{criterion}
Assume that $G$ preserves $Y$. Then $G$ preserves the symplectic form $\sigma _{F(Y)}$ on $F(Y)$ if and only if $G$ preserves 
$\mathit{Res} \frac{\Omega }{f_Y^2}$
\end{lem}

\section{First example}
We first recall results of \cite{Nam}. 
We define a smooth cubic 4-fold $Y$ by 
\[
Y:= \{ (z_0: \cdots :z_5) \in \bb P^5 |  f(z_0,z_1,z_2)+g(z_3,z_4,z_5)=0\} .
\] 
Define two elliptic curves by 
$$C:=\{(z_0: z_1:z_2) \in \bb P^2  | f(z_0,z_1,z_2)=0\}$$ and 
$$\ D:=\{ (z_3:z_4:z_5)\in \bb P^2 |   g(z_3,z_4,z_5)=0\}.$$
We put $$P_C=\{(z_0: \cdots :z_5) \in \bb P^5 | z_3=z_4=z_5=0  \},$$ 
and $$P_D=\{(z_0: \cdots :z_5) \in \bb P^5 | z_0=z_1=z_2=0  \}.$$ 
Notice that $C \subset P_C$ and $D \subset P_D$. 
Assume that $G \subset \mathit{PGL}(5)$ is the cyclic group $\bb Z_3$ of order $3$ generated by 
$$\tau = \begin{pmatrix}
1 &  &  &  &  &  \\
  & 1&  &  &  &  \\
  &  & 1&  &  &  \\
  &  &  &\zeta& & \\
  &  &  &   & \zeta& \\
  &  &  &   &  & \zeta 
\end{pmatrix}, $$ 
where $\zeta=\exp({2\pi \sqrt{-1}}/3)$. 

Let us consider the natural group action :
\[
\bb Z _3 {}^{\curvearrowright} \bb P ^5 \ni (z_0:z_1:z_2:z_3:z_4:z_5) \mapsto (z_0:z_1:z_2:\zeta z_3: \zeta z_4: \zeta z_5). 
\]
Since $\bb Z_3$ preserves $Y$, the action induces the $\bb Z_3$-action on $F(Y)$. 

\begin{lem}\label{crepant}
The induced action $\bb Z_3 {}^{\curvearrowright} F(Y)$ preserves the symplectic form $\sigma _{F(Y)}$, 
and the singular locus of $X_{\bb Z_3} := F(Y)/{\bb Z_3}$ is isomorphic to $C \times D$. 
The singularity is $A_2$. 
In particular $X_{\bb Z_3}$ has a crepant resolution $\nu : \tilde X_{\bb Z_3} \to X_{\bb Z_3}$. 
\end{lem}

\begin{proof}
By Lemma \ref{criterion}, 
the induced action $\bb Z_3 {}^{\curvearrowright} F(Y)$ preserves the symplectic form $\sigma _{F(Y)}$.
The singular locus of $X_{\bb Z_3}$ is isomorphic to the fixed locus $\mathrm{Fix}_{\bb Z_3}(F(Y))$ of the $\bb Z_3$ action on $F(Y)$. 
We remark that $\mathrm{Fix}_{\bb Z_3}(Y)$ is the disjoint union $C \cup D$. 
If a line $l$ is in $\mathrm{Fix}_{\bb Z_3}(F(Y))$, then $\bb Z_3$ acts on the line $l$. 
$l$ has two fixed points by the $\bb Z_3$ action. 
So $l \in F(Y)$ is in $\mathrm{Fix}_{\bb Z_3}(F(Y))$ if and only if 
$l$ passes through a point $p \in C$ and $q \in D$. Hence $\mathrm{Fix}_{\bb Z_3}(F(Y))$ is isomorphic to $C \times D$. 
		\begin{center}
			\includegraphics[height=5cm,width=8cm,clip]{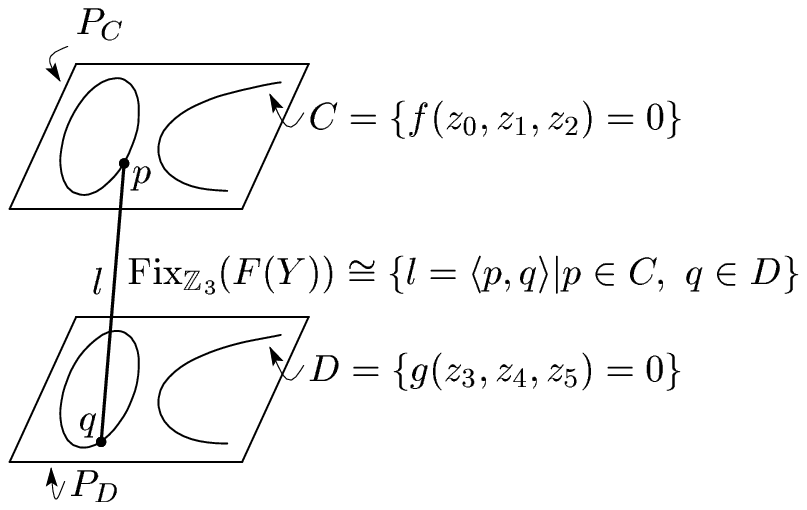}
		\end{center}

Since $\bb Z_3$ preserves $\sigma _{F(Y)}$, 
$X_{\bb Z_3}$ has $A_2$ singularities along $C \times D$. 
So, $X_{\bb Z_3}$ has a crepant resolution $\nu:\tilde X_{\bb Z_3} \to X_{\bb Z_3}$.
\end{proof}

\begin{prop}[\cite{Nam} 1.7.(iv)]\label{mainprop}
The notation being as above, 
$\tilde X_{\bb Z_3}$ is birational to $K^2(C \times D)$. 
\end{prop}

\begin{proof}
We first define a rational map $\psi : \tilde X_{\bb Z_3} \dashrightarrow K ^2 (C \times D)$. 
Let $l$ be in $F(Y)$, and $[l]$ the orbit of $l$ by the $\bb Z_3$ action : $[l]=\{ l, \tau (l) , \tau ^2(l) \}$.
Let $W_l$ be the linear space spanned by $l, \tau (l)$ and $ \tau ^2(l)$. 
If we choose $l \in F(Y)$ in general, then $W_l$ is isomorphic to $\bb P ^3$ and both $P_C \cap W_l$ and $P_D \cap W_l$ are lines. 
Furthermore we may assume that the cubic surface $S := W_l \cap Y$ is smooth. 
There are 27 lines in $S$ \cite{Har}. 
Notice that three lines $l , \tau (l)$ and $ \tau ^2 (l) $ are contained in $S$. 

We will prove follwoings. 
\begin{enumerate}
\item There are three lines $m_1,m_2$ and $m_3$ such that $m_i$ does not meet $m_j$ ($i \neq j$), and each $m_i $ meets $l $ at one point (respectively $\tau(l) , \tau ^2 (l)$). See also Figure \ref{biratmaps}. 
\item The set of three lines is unique, and  
\item each $m_i$ ($i=1,2,3$) is fixed by the $\bb Z_3 $ action. 
\end{enumerate}
Indeed the cubic surface $S$ is the blow up of $\bb P^2$ along 6 points $a_1, \ldots , a_6$. 
Let $E_i$ be the exceptional $(-1)$-curve over $a_i$, 
$C_i$ the proper transform of the conic passing through 5 points except $a_i$, and $L_{ij}$ the proper transform of line passing through $a_i$ and $a_j$. 
We may assume that $E_1= l $, $E_2= \tau(l)$ and $E_3= \tau^2(l)$. 
By the configuration of lines, each $m_i$ is given by $m_1 = C_4, m_2= C_5$ and $m_3=C_6$. 
There are exactly six lines which does not intersect $\tau^k(l)$ ($k=0,1,2$). 
These six lines are given by $L_{45}, L_{46}, L_{56}, E_4, E_5$ and $E_6 $. 
Since $\bb Z_3$ acts on the set $\{ E_1,E_2, E_3 \}$, 
$\bb Z_3$ acts on the set $\{L_{45}, L_{46}, L_{56} , E_4, E_5, E_6 \}$. 
Since the order of $\bb Z_3$ is $3$, $\bb Z_3$ acts on the two set $\{L_{45}, L_{46}, L_{56} \}$ and $\{E_4, E_5, E_6 \}$. 
If $E_4$, $E_5$ and $E_6$ are not fixed by $\bb Z_3 $ then all lines in $S$ move by the $\bb Z_3$ action. 
Since there are 9 lines in $S$ fixed by the $\bb Z_3$ action, 
three lines $E_4, E_5$ and $E_6$ are fixed by the $\bb Z_3$ action. 
Hence three lines $m_1,m_2$ and $m_3$ are fixed by the $\bb Z_3$ action. 
In particular $m_i$ ($i=1,2,3$) meet $C$ (resp $D$) at one point. 

Put $p_i= m_i \cap C$ and $q_i= m_i \cap D$. 
\begin{figure}[htbp]
\begin{center}
\includegraphics[clip,height=40mm]{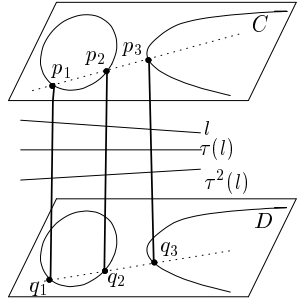}
\end{center}
\caption{Picture of $\psi$ and $\varphi$}
\label{biratmaps}
\end{figure}
Since three points $\{ p_1, p_2, p_3 \}$(resp. $\{ q_1, q_2, q_3 \}$) are in the line $P_C \cap W_l$, the sum $p_1 + p_2 + p_3$ is $0 \in C$ (resp. $q_1 + q_2 + q_3 = 0 \in D$). 
So we get an element $\{ (p_i, q_i) \}_{i=1}^{3}$ of $K^2(C\times D)$. 
We define $\psi $ as:
\[
\psi : \tilde X_{\bb Z_3} \dashrightarrow K ^2 (C \times D), \psi ([l]) =\{ (p_i, q_i) \}_{i=1}^{3}. 
\]  

We next define a rational map $\varphi : K^2(C \times D) \dashrightarrow \tilde X_{\bb Z_3}$. 
Let $\{ (p_i , q_i )  \}_{i=1}^3$ be in $K ^2(C \times D)$. 
We may assume that $p_i \neq p_j$ and $q_i \neq q_j$ ($i\neq j$). 
Let $m_i (i=1,2,3)$ be the line passing through $p_i$ and $q_i$. 
Let $M $ be the linear space spanned by $m_1, m_2$ and $ m_3$. 
Then $M$ is isomorphic to $\bb P ^3$. Let $S$ be the cubic surface $M \cap Y$. 
We remark that $S $ is smooth. 
By the configuration of 27 lines, 
there are exactly three lines $l_1,l_2$ and $l_3$ such that $l_i\cap l_j = \emptyset$ ($i\neq j$), and each $l_i $ meets $m_j $ ($i,j=1,2,3$) at exactly one point. 
We notice that $m_1,m_2$ and $m_3$ are fixed by the $\bb Z_3 $ action on $S$. 
By the uniqueness of the set of three lines, we have 
\[
\bb Z_3 {}^\curvearrowright \{ l_1, l_2, l_3 \} \mapsto \{ l_1, l_2, l_3\}.  
\]
Since each $l_i $ is not fixed by the $\bb Z_3$ action, we may assume that $\tau (l_1)=l_2$.   

By the uniqueness of the set of three lines, both $\varphi $ and $\psi$ are birational maps. 
In particular, ${\varphi}^{-1}=\psi$. 
\end{proof}

\subsection{On the indeterminacy of the birational maps}

In Theorem \ref{mainthm} we will prove that $\varphi $ and $\psi $ is given as the Mukai flop on 18 disjoint copies of $\bb P^2$. 
The key lemma of the proof is Lemma \ref{mainlem}. 
In the proof of Lemma \ref{mainlem}, 
we give the geometric description of the exceptional divisor of $\nu:\tilde X_{\bb Z_3} \to X_{\bb Z_3}$ 
by using $G$-Hilbert scheme. 
Let $X$ be a smooth projective variety over $\bb C$. We assume that a finite group $G$ acts on $X$. 
Then we define $G\mbox{-}\Hilb(X)$ as 
\[
G\mbox{-}\Hilb(X):= \{ Z \in \Hilb ^{\# G}(X) | \mca O_Z \cong \bb C[G]\ (\mathrm{as}\ G\mathrm{\ modules}) \}.
\]
and we call it \textit{$G$-Hilbert scheme of $X$}. 
Let $z$ be in $\tilde X_{\bb Z_3}$ 
and we assume that $z $ is not in the exceptional locus $\mathrm{Exc}(\nu)$ of $\nu$. 
Then $z$ is the orbit of $l \in F(Y)$ by the $\bb Z_3$ action. 
Hence $z$ is identified with $Z \in \bb Z_3 \mbox{-}\Hilb(F(Y))$. 
Assume that $z$ is in $\mathrm{Exc}(\nu)$. 
On a neighborhood of $\mathrm{Fix}(\bb Z_3)$, 
the $\bb Z_3$ action can be identified with
\[
\bb Z_3 \curvearrowright \mathrm{Spec}\bb C [x_1,x_2,x_3,x_4],\ (x_1,x_2,x_3,x_4) \mapsto (x_1,x_2,\zeta x_3, \zeta^2 x_4).
\] 
By using the $G$-graph, we see that $\bb Z_3\mbox{-}\Hilb(\mathrm{Spec}\bb C[x_1,x_2,x_3,x_4])$ is isomorphic to 
$\bb C^2\times \bb Z_3\mbox{-}\Hilb(\bb C^2)$ (See \cite{Nak}). 
So $\bb Z_3\mbox{-}\Hilb (\mathrm{Spec}\bb C[x_1,x_2,x_3,x_4]) $ is a crepant resolution of $\bb C^4/\bb Z_3$. 
Since this crepant resolution is unique, the element $z$ of $\tilde X_{\bb Z_3} $ can be identified with the element $Z$ of $\bb Z_3 \mbox{-}\Hilb (F(Y))$. 
By pulling back the universal family  $\mca L$ of $F(Y)$, 
we get the family of lines $\mca L_Z$ over $Z$. 
\[
\xymatrix{
\mca L_{Z}\ar[r] \ar[d] &\ar[d] \mca L \subset F(Y) \times \bb P^5  \ar[r] & \bb P^5 \\
Z      \ar[r]           & F(Y)
}
\]
The support of $\mca L_{Z}$ is $ l\cup \tau(l) \cup \tau^2 (l)  $, where $l \in  F(Y)$. 
Here we put $[l] := l\cup \tau(l) \cup \tau^2 (l)$. 
When $l$ is not in the fixed locus of the $\bb Z_3$ action, 
$\mca L_{Z}$ is uniquely determined by $[l]$. 
If $l$ is in the fixed locus of the $\bb Z_3$ action, then $Z$ is not determined uniquely by $[l]$. 

Let $p_2$ be the second projection 
$p_2:F(Y) \times \bb P^5 \to \bb P^5 $. 
For each $Z$, take the smallest linear subspace $L$ contained in $\bb P^5$ such that $\mca L_Z  \subset   p_2^{-1}(L) $.  
We call this $L$ the linear space spanned by $\mca L_Z$ (or simply  $Z$). 
In a generic case, $Z$ spans $\bb P^3$. In some special cases $Z$ spans $\bb P^2$.

\begin{lem}\label{mainlem}
Assume that the line $l \in F(Y)$ is in the fixed locus of the $\bb Z_3$ action, that is, $l$ passes through $p \in C$ and $q \in D$. 
We only consider $Z \in \tilde X_{\bb Z_3}$ such that the support of $\mca L_Z$ is $l$. 
Then $Z$ is parametrized by the tree $\mca I \cup \mca J$ of two projective lines $\mca I$ and $\mca J$. 

1. Assume that neither $p$ nor $q$ are 3-torsion points. 
Then $\mca L_{Z}$ spans $\bb P^3$, for each $Z$. 

2. Assume that both $p$ and $q$ are 3-torsion points. 
There exists exactly one point on each $\bb P^1$ such that 
$\mca L_Z$ spans $\bb P^2$. These two points are not in the intersection $\mca I \cap \mca J$. 
For other points, $\mca L_Z$ spans $\bb P^3$.

3. Assume that $p$ is a 3-torsion point and $q$ is not a 3-torsion point. 
There exists exactly one point on the tree such that $\mca L_Z$ spans $\bb P^2$. 
Moreover this point is not in the intersection $\mca I \cap \mca J$. 
Otherwise $\mca L_Z$ spans $\bb P^3$. 
\end{lem} 

When $q$ is a 3-torsion point and $p$ is not a 3-torsion point, we have a similar situation to \textit{3}. 
\begin{proof}
Let $Z \in \tilde X_{\bb Z_3}$ be a closed subscheme of length $3$ such that the support of $\mca L_Z$ is the line passing through $p \in C$ and $q \in D$. 
Let $(x:y:z:u:v:w)$ be the homogeneous coordinates of $\bb P^5$, 
consider the cubic $$ Y:=\{ (x: \cdots :w) \in \bb P^5 |   f(x,y,z) + g(u,v,w)=0  \}.$$ 
Let $C$ and $D$ be elliptic curves defined by $$C:=  \{ (x:y:z) \in \bb P^2 | f(x,y,z)=0 \},\  D:= \{ (u:v:w) \in \bb P^2 | g(u,v,w)=0 \}.$$ 
We put $$P_C := \{ (x: \cdots :w ) \in \bb P^5 |  u=v=w=0  \},$$ and 
$$P_D := \{ (x: \cdots :w) \in \bb P^5 | x=y=z=0  \}. $$

1. 
By changing coordinates, we may assume that 
$$p=(0:0:1), \ f(x,y,z)= x(x-z)(x- \alpha z) -y^2z ,\ \alpha \in \bb C- \{0,1 \} $$ 
$$q=(0:0:1),\ g(u,v,w)= u(u-w)(u-\beta w) -v^2w,\ \ \beta \in \bb C- \{0,1 \} . $$ 
So the line $l$ is given by:
\[
l= (0:0: \lambda : 0: 0: \mu ) \ \mathrm{where\ } (\lambda :\mu ) \in \bb P^1.
\]
Let $G(2,6)$ be the Grassmann manifold parameterizing $2$-dimensional subspace of $\bb C^6$. 
We will give coordinates system around $l$ by using the local coordinates of  $G(2,6)$. 
We define two hyperplanes 
\begin{eqnarray*}
\{ z=0 \}:= \{ (x: y:\cdots : w) \in \bb P^5| z=0 \}, \\
  \{ w=0 \}:= \{ (x: y:\cdots : w) \in \bb P^5| w=0 \}. 
\end{eqnarray*}
Then the intersection $\{ w=0 \} \cap l$ is $p=(0:0:1:0:0:0)$ and the intersection 
$\{z=0 \} \cap l$ is $q=(0:0:0:0:0:1)$. 
Let $(u_0,u_1,u_3,u_4)$ be affine coordinates of the hyperplane $\{ w =0\}$ around $p$ :
\[
 (u_0:u_1:1:u_3:u_4:0) \in \bb P^4  . 
\]
Let $(v_0,v_1,v_3,v_4)$ be affine coordinates of the hyperplane $\{ z =0\}$ around $q$ :
\[
 (v_0:v_1:0:v_3:v_4:1) \in \bb P^4   . 
\]
Let $U_l$ be an open neighborhood of $l$ in $G(2,6)$. The local coordinates of $U_l$ is given by $(u_0,u_1,u_3,u_4,v_0,v_1,v_3,v_4)$. 
Indeed, for each $(u_0,\cdots , v_4) \in U_l$, the corresponding line $\tilde l$ is given by 
\[
(u_0,\cdots , v_4) \leftrightarrow \tilde l:= \lambda (u_0:u_1:1:u_3:u_4:0) + \mu (v_0:v_1:0:v_3:v_4:1)
\]

The line $\tilde l \in G(2,6)$ is in $ F(Y)$ if and only if 
$\tilde l$ is contained in $ Y$. 
Therefore we have
\[
(f+g)(\tilde l) = \lambda ^3 F_1 + \lambda ^2 \mu F_2 + \lambda \mu ^2 F_3 + \mu ^3 F_4 =0
,\ \mathrm{for\ }\forall (\lambda :\mu ) \in \bb P^1 . 
\]
So we get the following equations
\begin{eqnarray*}
F_1 & := & u_0^3 -(\alpha +1) u_0^2 + \alpha u_0 - u_1 ^2 + u_3 ^3 =0 \\
F_2 & := & 3  u_0^2 v_0 -(\alpha +1 ) 2 u_0 v_0 + \alpha v_0 - 2 u_1 v_1 + 3 u_3 ^2 v_3 -(\beta +1 ) u_3 ^2 -u_4^2 =0\\ 
F_3 & := & 3  u_0 v_0^2  -(\alpha +1 )  v_0^2  - v_1^2 + 3 u_3 v_3^2 -(\beta +1 ) 2 u_3 v_3 + \beta u_3 -2 u_4 v_4 =0 \\
F_4 & := & v_0^3 +v_3^3 -(\beta +1) v_3^2 + \beta v_3  -v_4 ^2 =0 
\end{eqnarray*}
By the implicit function theorem, we can choose the affine coordinates of $F(Y)$ as $(u_1,u_4,v_1,v_4)$.

Let us consider the exceptional divisor of $\nu : \tilde X_{\bb Z_3} \to X_{\bb Z_3}$. 
Locally, the $\bb Z_3 $ action on $F(Y)$ can be written as 
\[
\bb Z_3 {}^\curvearrowright (u_1,u_4,v_1,v_4) \mapsto (u_1,\zeta u_4, \zeta ^2 v_1,v_4). 
\]
The exceptional divisor of $\bb Z_3 \mbox{-}\Hilb (\mathrm{Spec} \bb C[u_1,u_4,v_1,v_4]/\bb Z_3) $
 is given by $E_1+ E_2$, where 
\[ 
E_1 = \{ (u_1-\lambda, v_4-\mu, u_4^2, u_4v_1, v_1^3, au_4 + bv_1^2) | (\lambda,\mu) \in \bb C^2,\  (a:b) \in \bb P ^1 \} ,
\] 
\[
E_2 = \{ (u_1-\lambda, v_4-\mu, u_4^2, u_4^3, u_4 v_1, v_1^2, cu_4^2 + dv_1) | (\lambda,\mu) \in \bb C^2,\  (c:d) \in \bb P^1\}  .
\]
Here we identify the sub-schemes with their defining ideals. 
Since the line $l$ corresponds to the origin, 
when $Z\in E_1$, the ideal sheaf of $Z$ in $G(2,6)$ is 
\[
 I_{(a:b)} = (F_1, F_2, F_3, F_4 , u_1,v_4,u_4^2 ,u_4v_1,v_1^3, au_4 + bv_1^2) ,\  (a:b) \in \bb P^1.
\]
When $Z\in E_2$, the ideal sheaf of $Z$ in $G(2,6)$ is 
\[
 J_{(c:d)} = (F_1, F_2, F_3 , F_4 ,u_1 , v_4, u_4^3 , u_4v_1, v_1^3, cu_4^2 + dv_1) ,\  (c:d) \in \bb P^1 .
\]
By the calculation of generators of ideals, we have 
\[
 I_{(a:b)} = (u_0, v_3 ,v_0, \beta u_3-v_1^2, u_1, v_4, u_3^2, u_4v_1, v_1^3, au_4+bv_1^2),\  (a:b) \in \bb P^1 
\]
\[
 J_{(c:d)} = (u_0, u_3, v_3, \alpha v_0+u_4^2, u_1, v_4, u_4^3, u_4v_1, v_1^2, cu_4^2+d v_1),\  (c:d) \in \bb P^1 
\]
In particular, $Z$ is parametrized by $\mca I \cup \mca J $ where $\mca I= \{ I_{(a:b)} | (a:b) \in \bb P^1 \}$
 and $\mca J = \{  J_{(c:d)} | (c:d) \in \bb P^1   \}$. 
Assume that $Z$ is defined by $I_{(a:b)}$. 
The family of line $\mca L_{Z}$ over $Z$ is given by
\[
\mca L_{Z}  =  \{ (0: \mu v_1 : \lambda : \lambda u_3: \lambda u_4 : \mu ) 
| u_3^2=v_1^3=\beta u_3 -v_1^2=au_4+bv_1^2=u_4v_1=0 \} . 
\]
Since $\beta u_3- v_1^2 = a u_4 + bv_1^2=0$, we get an equation $a u_4 + b\beta u_3=0$. 
So, the linear space $L_{(a:b)}$ spanned by $\mca L_Z$ is  
\[
L_{(a:b)} = \{ (x:y:z:u:v:w) \in \bb P^5 | x= av + b \beta u =0 \}. 
\]
We remark that $L_{(a:b)}$ is spanned by the tangent line $\{ x=0 \} \subset P_C$ of $C$ at $p$ and the line $\{av + b \beta u =0  \} \subset P_D$ passing through $q \in D$. 

Assume that $Z$ is defined by $J_{(c:d)}$. 
In the same way, the linear space $M _{(c:d)}$ spanned by $\mca L_Z$ is  
\[
M_{(c:d)} = \{ (x:y:z:u:v:w) \in \bb P^5 | u= d y - c \alpha x  =0 \}. 
\]
$M_{(c:d)}$ is spanned by the line $\{ dy-c \alpha x =0 \} \subset P_C$ passing through $p \in C$ and the tangent line $\{ u=0 \} \subset P_D$ of $D$ at $q$. 
So we can draw the following picture. 
\begin{center}
\includegraphics[clip,width=12cm,height=8cm]{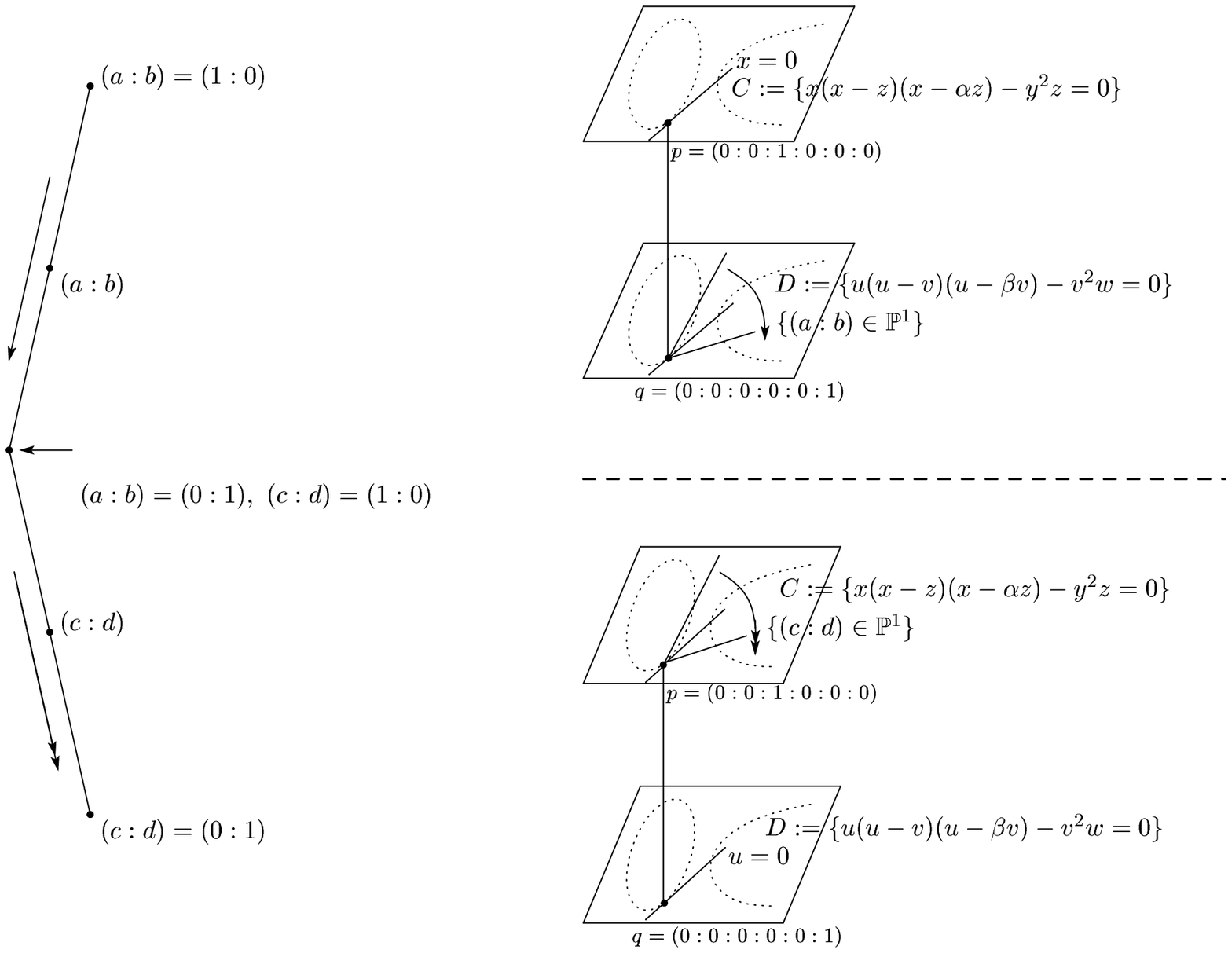}
\end{center}

2. By changing coordinates, we may assume that 
$$p=(0:1:0),\ f(x,y,z)= x(x-z)(x- \alpha z) -y^2z ,\ \alpha \in \bb C-\{0,1\} $$ and  
$$q=(0:1:0),\  g(u,v,w)= u(u-w)(u-\beta w) -v^2w ,\ \beta \in \bb C-\{0,1\} .$$ 

Let $U_l$ be a neighborhood of $l$ in $G(2,6)$. 
Local coordinates of $U_l $ are given by 
$(u_0,u_2,u_3,u_5,v_0,v_2,v_3,v_5) $. 
The corresponding line $\tilde l$ to $(u_0, \cdots ,v_5)$ is given by 
\[
(u_0, \cdots ,v_5) \leftrightarrow  \tilde l := \lambda (u_0:1:u_2:u_3:0:u_5) + 
\mu (v_0:0:v_2:v_3:1:v_5)
\]
Since $(f+g)(\tilde  l)$ should be zero for each $(\lambda :\mu ) \in \bb P^1$, we have 
\[
(f+g)(\tilde l)=\lambda ^3 F_1 + \lambda ^2\mu F_2 + \lambda \mu ^2 F_3 + \mu ^3 F_4 =0.   
\]
Let $d F_i$ be the exterior derivative of $F_i$ ($i=1, \cdots ,4$) at the origin. 
Then we have 
\[
\begin{pmatrix} dF_1 & dF_2 & dF_3 & dF_4  \end{pmatrix}
=
\begin{pmatrix} -du_2 & -dv_2 & -du_5 & -dv_5  \end{pmatrix}. 
\]
%
The local coordinates around $l \in F(Y)$ is given by $(u_0 , u_3, v_0, v_3)$. 
The $\bb Z_3$ action on $F(Y)$ is locally given by 
\[
\bb Z_3 {}^{\curvearrowright} F(Y) , (u_0,u_3,v_0,v_3 ) \mapsto (u_0,\zeta u_3, \zeta ^2 v_0,v_3) .
\]
The exceptional divisor of $\bb Z_3 \mbox{-}\Hilb(\mathrm{Spec}\bb C[u_0,u_3,v_0,v_3 ])$ is given by $E_1 + E_2$ where 
\[
E_1 = \{ ( u_0-\lambda ,v_3-\mu , u_3^2 , u_3v_0, v_0^3, a u_3 + bv_0^2) | (\lambda ,\mu)\in \bb C^2 ,(a:b) \in \bb P ^1 \}, 
\]
\[ 
E_2 = \{  (u_0-\lambda ,v_3-\mu ,u_3^3 , u_3 v_0, v_0^2, c u_3^2 + d v_0 ) |(\lambda,\mu) \in \bb C^2,  (c:d) \in \bb P^1 \}. 
\]
When $Z \in E_1$, the ideal sheaf of $Z$ in $G(2,6)$ is 
\[
 I_{(a:b)} = (u_2,u_5,  v_2 ,v_5, u_0, v_3 , u_3^2 , u_3v_0, v_0^3, au_3+bv_0^2) ,\  (a:b) \in \bb P^1 .
\]
When $Z \in E_2$, the ideal sheaf of $Z$ in $G(2,6)$ is 
\[
 J_{(c:d)} = (u_2, u_5, v_2, v_5, u_0, v_3, u_3^3, u_3v_0 , v_0^2 ,  cu_3^2+d v_0) ,\  (c:d) \in \bb P^1 . 
\]
In particular, $Z$ is parametrized by $\mca I \cup \mca J $ where $\mca I= \{ I_{(a:b)} | (a:b) \in \bb P^1 \}$
 and $\mca J = \{  J_{(c:d)} | (c:d) \in \bb P^1   \}$. 
Assume that $Z$ is defined by $I_{(a:b)}$. 
When $(a:b)\neq (1:0)$, $\mca L_{ Z}$ is 
\[
\mca L_{Z} =  \{ (\mu v_0: \lambda : 0 : \lambda u_3: \mu  : 0 ) |  u_3^2=u_3v_0=v_0^3=au_3+bv_0^2=0 \} . 
\]
The linear space $L_{(a:b)}$ spanned by  $\mca L_Z$ is 
\[
L_{(a:b)} = \{ (x:y:z:u:v:w) \in \bb P^5 | z= w =0 \}. 
\]
We remark that $L_{(a:b)}$ is spanned by the tangent line $\{z = 0 \} \subset P_C$ of $C$ at $p$ and the tangent line $\{  w=0\} \subset P_D$ of $D$ at $q$. 
When $(a,b )=(1,0)$, $\mca L_{Z}$ is
\[
\mca L_{Z} =  \{ (\mu v_0: \lambda : 0 : 0: \mu  : 0 ) |  v_0^3 =0 \}. 
\]
The linear space $L_{(1:0)}$ spanned by $Z$ is 
\[
L_{(1:0)} = \{ z=u=w=0 \} .
\]
The linear space $L_{(1:0)}$ is spanned by the tangent line $\{ z=0 \} \subset P_C$ of $C$ at $p$ and the point $q \in D$. 

Assume that $Z$ is defined by $J_{(c:d)}$. 
In the same way, when $(c:d)\neq(0:1)$, 
the linear space $M _{(c:d)} $ spanned by $Z$ is  
\[
M_{(c:d)} = \{ (x:y:z:u:v:w) \in \bb P^5 | z= w  =0 \}. 
\]
This linear space is spanned by the tangent line $\{ z=0 \} \subset P_C$ of $C $ at $p$ and the tangent line $\{w=0  \} \subset P_D$ of $D$ at $q$. 
When $(c:d )=(0:1)$, $\mca L_Z$ is 
\[
\mca L _Z = \{ (0: \lambda : 0 : \lambda u_3: \mu  : 0 ) |  u_3^3=0 \}. 
\]
So $M_{(0:1)}$ is $$\{z=w=x=0\}  .$$ 
The linear space $M_{(0:1)}$ is spanned by the point $p \in C$ and the tangent line $\{ w=0 \} \subset P_D$ of $D$ at $q$. 
\vspace{5pt}

3. 
By changing coordinates, we assume that 
$$p=(0:1:0),\ f(x,y,z)= x(x-z)(x- \alpha z) -y^2z,\ \alpha \in \bb C-\{0,1\}$$ and  
$$q=(0:0:1),\ g(u,v,w)= u(u-w)(u-\beta w) -v^2w,\ \beta \in \bb C-\{0,1\} .$$ 

Let $U_l$ be a neighborhood of $l$ in $G(2,6)$. 
Local coordinates of $U_l $ are given by
$ (u_0,u_2,u_3,u_4,v_0,v_2,v_3,v_4)$. 
For $(u_0, \cdots , v_4) \in U_l$, the corresponding line $\tilde l$ is given by 
\[
(u_0, \cdots , v_4) \leftrightarrow  \tilde l := \lambda (u_0:1:u_2:u_3:u_4:0) + 
\mu (v_0:0:v_2:v_3:v_4:1)
\]
When $\tilde l \in F(Y)$, we have $ (f+g)(\tilde l) = \lambda ^3F_1+ \lambda ^2 \mu F_2 + \lambda \mu ^2 F_3 + \mu ^3 F_4 =0$, for each $(\lambda , \mu) \in \bb P^1$. 
Let $dF_i$ be the exterior derivative of $F_i$ ($i=1, \cdots ,4$) at the origin. 
Then we have 
\[
\begin{pmatrix} dF_1 & dF_2 & dF_3 & dF_4  \end{pmatrix}
=
\begin{pmatrix} -du_2 & -dv_2 & \beta du_3 & \beta dv_3  \end{pmatrix}
\]
%
So the local coordinates of $F(Y)$ around $l$ are given by $(u_0 , u_4, v_0, v_4)$. 
The $\bb Z_3$ action is given by:
\[
\bb Z_3 {}^{\curvearrowright} F(Y) , (u_0,u_4,v_0,v_4 ) \mapsto (u_0,\zeta u_4, \zeta ^2 v_0,v_4) .
\]
The exceptional divisor of $\bb Z_3 \mbox{-}\Hilb (\mathrm{Spec}\bb C[u_0,u_4,v_0,v_4]) $ is given by $E_1+E_2$ where 
\[
E_1 = \{ ( u_0-\lambda ,v_4-\mu ,  u_4^2 , u_4v_0, v_0^3, a u_4 + bv_0^2) | (\lambda ,\mu) \in \bb C^2, (a:b) \in \bb P ^1 \} 
\]
\[ 
E_2 = \{  (u_0-\lambda ,v_4-\mu ,u_4^3 , u_4 v_0, v_0^2 , c u_4^2 + d v_0 ) | (\lambda ,\mu) \in \bb C^2, (c:d) \in \bb P^1\}.
\]
If $Z\in E_1$ then $Z$ is defined by 
\[
 I_{(a:b)} = (u_2,u_3,  v_2 ,v_3, u_0, v_4 , u_4^2 , u_4v_0, v_0^3, au_4+bv_0^2) ,\  (a:b) \in \bb P^1 
\]
If $Z\in E_2$ then the ideal sheaf of $Z$ is  
\[
 J_{(c:d)} = (u_2, u_3, v_3, u_4^2+v_2, u_0, v_4, u_4^3, u_4v_0 , v_0^2 ,  cu_4^2+d v_0) ,\  (c:d) \in \bb P^1.  
\]
In particular, $Z$ is parametrized by $\mca I \cup \mca J $ where $\mca I= \{ I_{(a:b)} | (a:b) \in \bb P^1 \}$
 and $\mca J = \{  J_{(c:d)} | (c:d) \in \bb P^1   \}$. 
Assume that $Z$ is defined by $I_{(a:b)}$. 
When $(a:b)\neq(1:0)$, $\mca L_{Z}$ is 
\[
\mca L_{Z} =\{ (\mu v_0: \lambda : 0 : 0 : \lambda u_4  : \mu ) |  u_4^2=u_4v_0=v_0^3=au_4+bv_0^2=0 \} . 
\]
The linear space $L_{(a:b)}$ spanned by $\mca L_Z$ is 
\[
L_{(a:b)} = \{ z=u=0 \}. 
\]
This linear space is spanned by the tangent line $\{z=0\} \subset P_C$ of $C$ at $p$ and  the tangent line $\{ w=0 \} \subset P_D$ of $D$ at $q$. 
When $(a:b)=(1:0)$, $\mca L_{Z}$ is 
\[
\mca L_{Z} = (\mu v_0: \lambda : 0 : 0 : 0  : \mu ) |  v_0^3=0  \} . 
\]
Then the linear space $L_{(1:0)}$ spanned by $\mca L_Z$ is 
\[
L_{(1:0)} = \{ z=u=v=0 \} .
\] 
This linear space is spanned by the tangent line $\{ z=0\} \subset P_C $ of $C$ at $p$ and the point $q \in D$. 

Assume that $Z$ is defined by $J_{(c:d)}$. 
For each $(c:d) \in \bb P^1$, $\mca L_{Z}$ is 
\[
\mca L_{Z} = (\mu v_0: \lambda : \mu v_2 : 0 : \lambda u_4  : \mu ) | u_4^2+v_2 = u_4^3 = u_4v_0 = v_0^2 = cu_4^2+dv_0=0 \} . 
\]
Similarly, for each $(c:d) \in \bb P^1$, the linear space $M _{(c:d)}$ spanned by $\mca L_Z$ is 
\[
M_{(c:d)} = \{ (x:y:z:u:v:w) \in \bb P^5 | u= cx -dz  =0 \}. 
\]
The linear space is spanned by the line $\{ cx-dz=0 \} \subset P_C$ passing through the point $p$ and 
the tangent line $\{ u=0  \} \subset P_D$ of $D$ at $q$. 
\end{proof}

To describe the indeterminacy of $\psi$ explicitly, 
we define $Q_{(\mathrm{I})}$ and $Q_{(\mathrm{II})}$ as:
\[
Q_{(\mathrm{I})} := \{ Z \in \tilde X_{\bb Z_3} | \mca L_Z\mathrm{\ spans\ }\bb P^2\mathrm{\ and\ } \mca L_Z \ni p \mathrm{\ where\ }p \in C\ \mathrm{and}\ 3p=0   \},
\]
and 
\[
Q_{(\mathrm{II})} := \{ Z \in \tilde X_{\bb Z_3} | \mca L _Z\mathrm{\ spans\ }\bb P^2\mathrm{\ and\ } \mca L_Z\ni q \mathrm{\ where\ }q \in D\ \mathrm{and}\ 3q=0   \}. 
\]
Let $ Q :=Q_{(\mathrm{I})}  \cup Q_{(\mathrm{II})} $. 
We shall prove that the indeterminacy of $\psi $ is $Q$ in Theorem \ref{mainthm}. 
\begin{lem}\label{Q}
If $\mca L_Z$ spans $\bb P^2$, then $\mca L_Z$ passes through a 3-torsion point of $C$ or $D$. 
Furthermore $Q_{(\mathrm{I})} $ is isomorphic to $\{ p\in C| 3p=0 \} \times P_D$, and 
$Q_{(\mathrm{II})} $ is isomorphic to $\{ q \in D| 3q=0 \} \times P_C$. 
In particular, $Q_{(\mathrm{I})} $ and $Q_{(\mathrm{II})}$ are respectively isomorphic to 9 disjoint copies of $\bb P^2$. 
\end{lem}

\begin{proof}
We first prove the first assertion. 
Let $l$ be in $F(Y)$. 
Since the support of $\mca L_Z$ is $l\cup \tau(l) \cup \tau^2(l)$, we have to consider the following two cases. 

\begin{enumerate}
\item $l,\tau (l)$ and $\tau ^2(l)$  span $\bb P^2$. 
\item  $l=\tau (l) = \tau ^2(l)$, that is, $l$ is in the fixed locus of the $\bb Z_3$ action. 
\end{enumerate}

1. Let us consider the first case. 
Let $P_l := \< l,\tau(l), \tau ^2(l) \>$ be the plane spanned by $l,\tau(l)$ and $\tau ^2(l)$. 
Then $\bb Z_3$ acts on $P_l$. The eigenvalue of the $\bb Z_3$ action on $\bb P^5$ is 1 or 
$\zeta = \exp (\frac{2 \pi \sqrt{-1}}{3})$. 
So the eigenvalue of the $\bb Z_3$ action on the plane $P_l$ is also 1 or $\zeta$. 
Therefore we may assume that, for suitable coordinates of $P_l$, 
the group action $\bb Z_3{}^{ \curvearrowright} P_l$ is given by 
\[
\bb Z_3 {}^{\curvearrowright} P_l, (x:y:z) \mapsto (x:y:\zeta z).
\]
Therefore the fixed locus $\mathrm{Fix}_{\bb Z_3}(P_l)$ of the $\bb Z_3$ action on $P_l$ is the disjoint union of a line $m_{\mathrm{fix}}$ and a point $q_{\mathrm{fix}}$. 
Since $\mathrm{Fix}_{\bb Z_3}(\bb P^5)$ is $P_C \cup P_D$, 
we may assume that $m_{\mathrm{fix}} = P_l \cap P_C,\ q_{\mathrm{fix}} = P_l \cap P_D$. 
Put $p=l \cap m_{\mathrm{fix}}$. Then $p$ is contained in $\mathrm{Fix}_{\bb Z_3}(Y)$. 
Since $\mathrm{Fix}_{\bb Z_3}(Y)=C \cup D$, $p \in C$. 
We will prove that the point $p$ is a 3-torsion point of $C$. 
To prove this, 
it is enough to show that $m_{\mathrm{fix}}$ meets $C$ only at the point $p$. 
So we assume that there exists another point $p^{\prime }$ on $C \cap m_{\mathrm{fix}}$. 
However $P_l \cap Y$ must be three lines and these three lines meet at exactly one point $p$. 
Since $p^{\prime}$ is in $Y$, this is contradiction. 
Therefore $p$ is a 3-torsion point of $C$. Hence $l$ passes through a 3-torsion point of $C$.

2. Let us consider the second case. We remark that this case is the limit of the first case. 
Let $Z $ be in $\tilde X_{\bb Z_3}$ such that the support of $\mca L_Z$ is $l$. 
By the proof of Lemma \ref{mainlem}, 
$\mca L_Z$ which spans $\bb P^2$ always passes through a 3-torsion point. 
Hence we get the first assertion.  

We shall prove the second assertion for $Q_{(\rm{I})}$. 
Assume that $Z\in Q_{(\rm{I})}$. By the above argument, the plane $\< \mca L_Z \>$ spanned by $\mca L_Z$ intersect $P_D$ at one point and 
$\<\mca L_Z\> \cap P_C$ is the tangent line of $C$ at a 3-torsion point. 
Conversely we choose a 3-torsion point $p$ of $C$ and a point $q$ from the plane $P_D$. 
Let $l_C$ be the tangent line of $C$ at $p$, and $\< l_C, q  \> $ the plane spanned by $l_C$ and $q$.  
There uniquely exists $Z\in \tilde X_{\bb Z_3}$ such that $\mca L_Z$ spans $\<l_C, q  \>$. 
This correspondence gives isomorphism between $Q_{(\rm{I})}$ and $\{p \in C | 3p=0 \} \times P_D$. 
Therefore $Q_{(\mathrm{I})}$ is isomorphic to $\{p \in C | 3p=0 \} \times P_D$. 
Similarly we can show that $Q_{(\mathrm{II})}$ is isomorphic to $\{q \in D | 3q=0 \} \times P_C$. 
\end{proof}

To describe the indeterminacy of $\varphi:K^2(C \times D) \dashrightarrow \tilde X_{\bb Z_3}$, 
we put 
\[
\bb P_{(\mathrm{I})}:=\Bigl\{ \{(p,q_1),(p,q_2),(p,q_3) \} \in K^2 (C \times D) \Big| 3p=0 \Bigr\}, 
\]
and 
\[
\bb P_{(\mathrm{II})}:=\Bigl\{ \{(p_1,q),(p_2,q),(p_3,q) \} \in K^2 (C \times D) \Big| 3q=0 \Bigr\}. 
\]
We remark that 
$$\bb P_{(\rm{I})} \cong \{ p \in C | 3p=0 \} \times P_D^{\vee},$$ and 
$$\bb P_{(\rm{II})} \cong \{ p \in C | 3p=0 \} \times P_D^{\vee},$$ 
where  $P_D^{\vee} $ (resp. $P_C^{\vee}$) is the dual of $P_D$ (resp. $P_C$). 
Moreover, $\bb P_{(\mathrm{I})}$ and $\bb P_{(\mathrm{II})}$ are respectively isomorphic to 9 disjoint copies of $\bb P^2$. 

\begin{thm}\label{mainthm} 
The birational map $\psi:\tilde X_{\bb Z_3}  \dashrightarrow K ^2 (C \times D)$ 
is decomposed into the Mukai flop on $Q$, 
and $\varphi$ is decomposed into the Mukai flop on $\bb P_{(\mathrm{I})}\cup \bb P_{(\mathrm{II})}$. 
In particular, the indeterminacy of $\psi $ is $Q$, and 
the indeterminacy of $\varphi $ is $\bb P_{(\mathrm{I})}\cup \bb P_{(\mathrm{II})}$. 
\end{thm}

The proof of Theorem \ref{mainthm} is long. 
Before starting the proof, we would like to explain the strategy of it.

\begin{step}\label{1}
We will blow up $\tilde X_{\bb Z_3}$ along $Q$. We denote it by $\pi:W \to \tilde X_{\bb Z_3}$. 
Now let us consider the situation when an irreducible symplectic 4-fold $X$ contains $P$ which is isomorphic to $\bb P^2$. 
Since $\bb P^2$ is a Lagrangian manifold, the normal bundle $N_{P/X}$ of $P$ in $X$ is isomorphic to 
the cotangent bundle $\Omega_{P}^1$ of $P$. 
Let ${\it B}_{P} X \to X$ be the blow up along $P$. Then the exceptional locus is isomorphic to 
\[
{\it B}_{P} X \supset \Gamma := \{ (x,H) \in P \times P^{\vee} | x \in H \} \to P \subset X,
\]
where $P^{\vee}$ is the dual of $P$. 
By Fujiki-Nakano contraction theorem, there exists a birational contraction morphism ${\it B}_{P} X \to X^+$ such that 
the exceptional locus is isomorphic to $\Gamma \to P^{\vee}$. 
So, we will blow down $W$ in another direction $\pi ^+ : W \to W^{+}$. 
Here $W^+$ is a Moishezon 4-fold. We shall prove that $W^+$ is actually a projective manifold in Proposition 3.11.
\end{step}

\begin{step}\label{2}
We will define a morphism $\mu^{\prime }: W \to \bar K^2(C \times D)$. 
We will prove that $\mu ^{\prime} $ factors through $W^+$. 
Namely there is a map $\mu ^+: W^+ \to \bar K^2(C \times D)$ such that $\mu ^+ \circ \pi ^+ = \mu ^{\prime}$. 
\end{step}
 
\begin{step}\label{3}
We will prove that $W^+$ is isomorphic to $K^2(C \times D)$. 
The following diagram illustrates three steps above. 
\[
\xymatrix{
& W\ar[rdd]_{\mu^{\prime}} \ar[rd]^{\pi^{+}} \ar[ld]_{\pi} &   \\
\tilde X_{\bb Z_3}  \ar@{--}[r] & \ar@{-->}[r]  & W^{+}  \ar[d]^{\mu ^+}   & K^2(C \times D) \ar[dl]_{\mu} \\
                                                        &                     &            \bar K^2(C \times D) & 
}
\]
\end{step}

Let us consider Step \ref{1}. 
We first construct the blow up $\pi : W \to \tilde  X_{\bb Z_3}$. 

\begin{prop}\label{blowup}
Let $G(4,6)$ be the Grassmann manifold which parametrizes 4 dimensional linear subspaces of $\bb C ^6$. 
We define $W$ by  
\begin{multline*}
W := \{ (Z, L)  \in \tilde X_{\bb Z_3} \times G(4,6)  
      | \mca L _Z \subset L,\  \\
    \mathrm{both}\ L\cap P_C\mathrm{\ and\ }L\cap P_D\mathrm{\ are\ lines} \}. 
\end{multline*}
Then the natural projection $\pi :W \to \tilde X_{\bb Z_3}$ is the blow up along $Q$. 
\end{prop}

\begin{proof}
Let $\pi^{\prime} : \it{B}_Q\tilde X_{\bb Z_3} \to \tilde X_{\bb Z_3}$ be 
the blow up of $\tilde X_{\bb Z_3} $ along $Q$. 
Let $\mathrm{Exc}(\pi ^{\prime })$ be the exceptional divisor of $\pi^{\prime}$ and $\mathrm{Exc}(\pi)$ the exceptional divisor of $\pi $. 

We will define the morphism $f: \it{B}_{Q}\tilde X_{\bb Z_3} \to W$. 
Outside of the exceptional divisors, we define $f$ by 
$f:=\pi^{-1} \circ \pi ^{\prime}$. 
On $\mathrm{Exc}(\pi^{\prime})$, we define $f$ in the following way. 
On $(\pi')^{-1}(Q_{(\mathrm{I})})$, the morphism $\pi ^{\prime}$ is defined as follows: 
\begin{eqnarray*}
\{ (p,x,H) \in  C \times  P_D \times P_D^{\vee} |3p=0, x \in H \} & \stackrel{\pi'}{\to} & \{ (p,x) \in C\times P_D | 3p=0 \}\\ 
\pi^{\prime}(p,x,H) &=& (p,x)
\end{eqnarray*}
By Lemma \ref{Q}, $(p,x)$ determine $Z_{(p,x)} \in Q_{(\rm{I})}$ uniquely. 
Here let $l_{C,p}$ be the tangent line of $C$ at $p$. 
Hence we define a morphism $f :\mathrm{Exc}(\pi^{\prime}) \to \mathrm{Exc}(\pi)$ by 
$f(p,x,H) = (Z_{(p,x)},  \< l_{C,p}, H \>  )$ where $\< l_{C,p}, H \>$ is the projective linear space spanned by $l_{C,p}$ and $H$. 
In the same way, we can define $f$ on $(\pi')^{-1}(Q_{(\mathrm{II})})$. 
Since $f$ is bijective, $W$ is the blow up along $Q$. 
\end{proof}

As we remarked, 
by Fujiki-Nakano contraction theorem, there exists a contraction morphism 
$\pi ^+ : W \to W^+$. 
Note that on $\mathrm{Exc}(\pi ^+)$, $\pi^+$ can be identified as  
\[
W \supset \mathrm{Exc}(\pi^+) \ni (Z,L) \to L \in G(4,6). 
\]
We choose $w \in \pi ^+({\rm Exc}(\pi^+))$ and fix it. 
We remark that for each $(Z,L) \in (\pi ^+)^{-1}(w)$, $L$ is uniquely determined by $w$.

Next let us consider Step \ref{2}. 
We first define $\mu ^{\prime} : W  \to \bar K^2(C\times D)$ by
\[
\mu ^{\prime } (Z,L) := \{ (p_i,q_i) \}_{i=1}^3\ \mathrm{where\ } \< p_i, q_i \> \subset L \cap Y,\ \< p_i,q_i \> \cap \mca L_Z \neq \emptyset . 
\]
The map $\mu^{\prime}$ is well-defined. 

\begin{lem}\label{factor}
The map $\mu^{\prime}$ factors through $W^+$, 
that is, there is a map $\mu ^+:W^+ \to \bar K^2(C \times D)$ such that 
$\mu ^+ \circ \pi ^+ = \mu ^{\prime}$. 
\end{lem}

\begin{proof}
Let $w $ be a point of $W^+$. 
We define $\mu ^+ $ in the following way:
\[
\mu ^{+}(w):=\mu^{\prime }\bigl( (\pi^+)^{-1}(w) \bigr). 
\]
If $\mu ^+$ is well-defined then the assertion of Lemma \ref{factor} is clear. 
So, we will prove that $\mu^{\prime }\bigl( (\pi^+)^{-1}(w) \bigr)$ is independent of the fiber $(\pi^+)^{-1}(w)$. 
If $w $ is not in $\pi ^+ (\pi ^{-1}(Q))$, then $(\pi ^+)^{-1}(w) \cap \pi ^{-1}(Q) = \emptyset$. 
Since the fiber is unique, $\mu ^+$ is well-defined. 

Assume that $w $ is in $\pi^+( \pi^{-1}(Q))$. 
Let $(Z,L) \in (\pi^+)^{-1}(w)$. 
As we remarked, 
the linear space $L$ is uniquely determined by $w$. 
Then exactly one of the following two cases will happen. 
\begin{enumerate}
\item $L \cap P_C$ is a tangent line of $C$ at a 3-torsion point. 

\item $L \cap P_D$ is a tangent line of $D$ at a 3-torsion point. 
\end{enumerate}
 
We may assume that the first case occurs. 
By Proposition \ref{blowup} $(\pi^+)^{-1}(w)$ is parametrized by the line $L \cap P_D$. 
Assume that $L \cap P_D \cap D$ consists of three points $\{ q_1,q_2,q_3 \}$. 
For each $(Z,L) \in (\pi^+)^{-1} (w) $, 
\[
\mu ^{\prime }(Z, L) =  \{ (p,q_1), (p,q_2), (p,q_3) \} \in \bar K^2(C \times D).
\] 
Therefore the map $\mu ^+$ is well-defined. 
\end{proof}

We consider Step \ref{3} to complete the proof of Theorem \ref{mainthm}. 
We will prove that the map $\mu^+$ is isomorphism on the smooth locus of $\bar K^2(C \times D)$ in the following Lemma.

\begin{lem}\label{smooth}
The map $\mu ^+$ is an isomorphism on the smooth locus of $\bar K^2(C \times D)$. 
\end{lem}

\begin{proof}
Take $\{(p_i,q_i)  \}_{i=1}^{3}$ from the smooth locus of $\bar K^2(C \times D)$. 
We will prove that the fiber $(\mu ^+)^{-1}(\{(p_i,q_i)  \})  $ (for short $\{(p_i,q_i)\}_{\mu^+}$) is one point. 
Let $l_C \subset P_C$ (resp. $l_D \subset P_D$) be the line passing through three points $\{ p_i \}_{i=1}^3$ (resp. $\{  q_i \}_{i=1}^3$). 
We have to consider the fibers of $\mu ^+$ in the following cases. 
\begin{enumerate}
\item $ \bar K^2(C \times D) \ni  \{ (p_i,q_i)  \}_{i=1}^{3}$ 
where $p_i \neq p_j$ $(i\neq j)$, and $q_i \neq q_j$ $(i\neq j)$. 
\item $ \bar K^2(C \times D) \ni  \{ (p_i,q_i)  \}_{i=1}^{3}$ where $p_1=p_2$, $p_1 \neq p_3$ and 
$q_i \neq q_j$ $(i\neq j)$. 
\item $ \bar K^2(C \times D) \ni  \{ (p_i,q_i)  \}_{i=1}^{3}$ where $p_1=p_2$, $p_1 \neq p_3$ and 
$q_2=q_3$, $q_2 \neq q_1$. 
\item $ \bar K^2(C \times D) \ni  \{ (p_i,q_i)  \}_{i=1}^{3}$ where $p_1=p_2=p_3=:p$, and $q_i \neq q_j$ $(i\neq j)$. 
\end{enumerate}

1. Let us consider the first case. 
In this case, the fiber $\{(p_i,q_i)\}_{\mu^+}$ is isomorphic to the fiber 
$(\mu^{\prime})^{-1}(\{ (p_i,q_i) \})$ (for short $\{(p_i,q_i)\}_{\mu '}$). So assume that $(Z,L) \in \{(p_i,q_i)\}_{\mu '}$. 
The linear space $L$ is spanned by $l_C$ and $L_D$, and 
the cubic surface $L \cap Y$ is smooth. 
As we wrote in Proposition \ref{mainprop}, $Z$ consists of three lines $l$, $\tau(l)$ and $\tau^2(l)$. 
In particular $Z$ is uniquely determined. 
Hence the fiber is one point.

2 and 3. Notice that the case 3 is the limit of the case $2$. 
Since $\{(p_i,q_i)\}_{\mu '} \cap \pi ^{-1}(Q) = \emptyset  $, 
the fiber $\{(p_i,q_i)\}_{\mu '}$ is isomorphic to the fiber $\{(p_i,q_i)\}_{\mu ^+}$. 
Let $  (Z,L) \in  \{(p_i,q_i)\}_{\mu'}$.  
The linear space $L$ is spanned by $l_C$ and $l_D$. 
Let us consider the cubic surface $L \cap Y$. 
There is a line $l \subset L \cap Y$ which intersect three lines $ \{ \< p_i, q_i \> \}_{i=1}^{3}$. 
In both case $2$ and $3$, 
the line $l$ passes through $p_1$ and $q_3$. 
Notice that the line $l$ is in the fixed locus of the $\bb Z_3$ action on $F(Y)$. 
Since there uniquely exists $Z$ such that $\mca L_Z$ spans $L$, 
the fiber is one point.  

4. Let us consider the fourth case. 
Assume that $(Z,L) \in \{(p_i,q_i)\}_{\mu '} $.  
By the similar argument in the proof of Lemma \ref{factor}, 
$\{(p_i,q_i)\}_{\mu '}$ is parametrized by the line $l_D$. 
Since $\{(p_i,q_i)\}_{\mu '}$ is contracted to a point by $\pi ^+$, 
the fiber $\{(p_i,q_i)\}_{\mu ^+}$ is one point. 
\end{proof}

We will use the following result to complete Step \ref{3}. 
\begin{prop}[\cite{F-N}]
Let $Y$ be a locally $\bb Q$ factorial symplectic variety, and suppose that 
there is a projective crepant resolution of $Y$: $\pi : X \to Y$. 
If the exceptional divisor of $\pi$ is irreducible, then $\pi$ is a unique crepant resolution, that is, 
if there is another crepant resolution $\pi' : X' \to Y$ then the natural birational map $\pi' \circ \pi^{-1}$ is an isomorphism. 
\end{prop}
Since the exceptional divisor of $\mu :K^2(C \times D) \to \bar K^2(C \times D)$ is irreducible, 
it is enough to show that $W ^+$ is projective. 
We will observe the fiber of singular points of $\bar K^2(C \times D)$ by $\mu^+$ and 
prove that $W ^+$ is projective
\begin{prop}
Let $E$ be the exceptional divisor of $\mu ^+ : W ^+ \to \bar K^2(C \times D)$. 
Then $-E$ is $\mu ^+$-ample. In particular $W^+$ is a projective manifold. 
\end{prop}

\begin{proof}
We first consider the fiber of singular points by $\mu ^+$. 
By Lemma \ref{smooth}, the exceptional divisor of $\pi ^+$ is contained in the inverse image of 
the singular locus of $\bar K^2(C \times D)$. 
We will use the same notation in the proof of Lemma \ref{smooth}. 
By symmetry we have to consider the following three cases. 
\begin{enumerate}
\item $\{ (p_i,q_i) \}_{i=1}^{3} \in \bar K^2(C \times D)$. 
$p_1=p_2=:p$, $p_1 \neq p_3 =-2p$ and $q_1=q_2=:q$, $q_1 \neq q_3 =-2q$. 
\item $\{ (p_i,q_i) \}_{i=1}^{3} \in \bar K^2(C \times D)$.  
$p_1=p_2=p_3 =:p$ and $q_1=q_2=:q$, $q_1 \neq q_3 =-2q$. 
\item $\{ (p_i,q_i) \}_{i=1}^{3} \in \bar K^2(C \times D)$. 
$p_1=p_2=p_3 =:p$ and $q_1=q_2=q_3 =:q$ 
\end{enumerate}

1. 
Let us consider the first case. 
We claim that the fiber $\{(p_i,q_i)\}_{\mu^+}$ is isomorphic to $\bb P^1$. 
Since $\{(p_i,q_i)\}_{\mu'} \cap  \pi^{-1}(Q) = \emptyset$, 
we have $\{(p_i,q_i)\}_{\mu^+} \cong \{(p_i,q_i)\}_{\mu'}$. 
So, let $(Z,L) \in \{(p_i,q_i)\}_{\mu'}$. 
Since the linear space $L$ is spanned by the line $l_C$ and $l_D$, $L$ is unique. 
There are infinitely many lines in the cubic surface $L \cap Y$ which intersect with two lines 
$ \< p, q \>$ and $\<-2p, -2q \> $. 
Here we can draw the following picture of the configuration of lines in $L \cap Y$. 
\begin{center}
\includegraphics[clip,height=40mm]{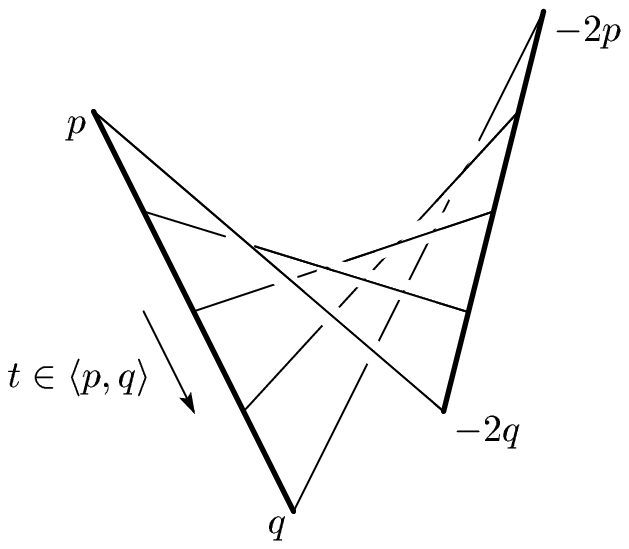}
\end{center}
The set of lines can be parametrized by the line $\<p,q\>$, that is, 
the line $l$ in $ L \cap Y$ is uniquely determined by $t \in \< p,q\>$. 
The cubic surface has singular points along the line $\< p,q\>$.

Notice that $\bb Z_3$ acts on the line $\< p,q \>$. 
The fiber $\{(p_i,q_i)\}_{\mu'}$ is isomorphic to $\< p,q\>/\bb Z_3$. 
Indeed let $l_t\ \subset L\cap Y$ be the line passing through $t \in \< p,q\>$. 
When $t \in \< p,q\> -\{p,q \}$, 
$l_t$ is not in the fixed locus of the $\bb Z_3$ action on $F(Y)$. 
Hence there exists uniquely $Z $ in $\tilde X_{\bb Z_3}$ such that the support of $\mca L_Z$ is $l_t \cup \tau (l_t) \cup \tau^2(l_t)$. 
When $t = p$, $l_p$ passes through $p$ and $-2q$. 
Notice that $l_p$ is in the fixed locus of the $\bb Z_3$ action on $F(Y)$. 
By Lemma \ref{mainlem}, there uniquely exists $Z \in \tilde X_{\bb Z_3}$ such that the support of $\mca L_Z$ is 
$l_p $ and $\mca L_Z$ spans $L$. 
Similarly, if $t = q$, then the line $l_q$ passes through $-2p$ and $q$. 
By $L$, $Z$ can be determined uniquely. 
So $\{(p_i,q_i)\}_{\mu^+}$ is isomorphic to $\< p,q\>/\bb Z_3$.

2. 
Let us consider the second case. 
Recall the notation $\mca I$ defined in third case of the proof for Lemma \ref{mainlem}. 
We will prove that the fiber $\{(p_i,q_i)\}_{\mu^+}$ is parametrized by $\mca I$. 
Let $ (Z,L)  \in \{(p_i,q_i)\}_{\mu'}$. 
Then $L$ is spanned by $l_C$ and $l_D$. 
Furthermore, the cubic surface $L \cap Y$ is given by $\{ (x:y:z:w) \in L |y^3 +z^2w =0\}$ for suitable homogeneous coordinates of $L$. 
The point $(1:0:0:0) \in L$ correspond to $p$. 
Notice that all lines contained in the cubic surface $L \cap Y$ pass through $p$. 

We claim that the fiber $\{(p_i,q_i)\}_{\mu'}$ is isomorphic to the tree $\mca I \cup l_D$ of two projective line $\mca I$ and $l_D$. 
Let $P_t$ be the plane spanned by $l_C$ and $t \in l_D$. 
If $t \in l_D-\{q,-2q \}$, then $P_t \cap Y$ is not in the fixed locus. 
So we can uniquely determine $Z$ such that the support of $\mca L_Z$ is $P_t \cap Y$. 
When $t = -2q$, $P_t \cap Y$ is the line $ \< p,-2q\> $ passing through $p$ and $-2q$. 
There uniquely exists $Z$ such that the support of $\mca L_Z$ is the line $\< p,-2q\>$ and $\mca L_Z$ spans $L$. 
When $t = q$, $P_t \cap Y$ is the line $\< p,q\>$. 
Here recall that $I_{(a:b) } \in \mca I$ is the ideal defined in the third case in the proof of Lemma \ref{mainlem}. 
Let $Z_{(a:b) } \in  \tilde X_{\bb Z_3}$ be the point defined by $I_{(a:b)}$. 
For each $ I_{(a:b)} \in \mca I$, $(Z_{(a:b)},L)$ is in the fiber $\{(p_i,q_i)\}_{\mu'}$. 
So in the second case, we can not uniquely determine $Z$. 
Therefore, $\{(p_i,q_i)\}_{\mu'}$ is isomorphic to $\mca I \cup l_D$. 

Since the line $l_D$ is contracted to a point by $\pi ^+ :W \to W^+$, 
the fiber $\{(p_i,q_i)\}_{\mu^+}$ is isomorphic to $\mca I$. 
We draw the picture of the ``blow down''. 
\begin{center}
\includegraphics[clip]{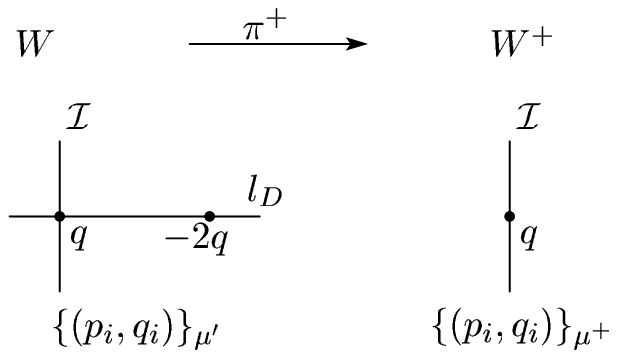}
\end{center}

3. 
Let us consider the third case. 
When $(Z,L) \in  \{(p_i,q_i)\}_{\mu'}$, the linear space $L$ is spanned by $l_C$ and $l_D$. 
Moreover the cubic surface $L \cap Y$ is three planes $P_1$, $P_2$ and $P_3$. 
Each plane is sent to another plane by the $\bb Z_3$ action. 
Three planes share the line $\< p,q\>$. So we choose a plane $P_1$ and fix it. 
\begin{center}
\includegraphics[clip,height=4cm]{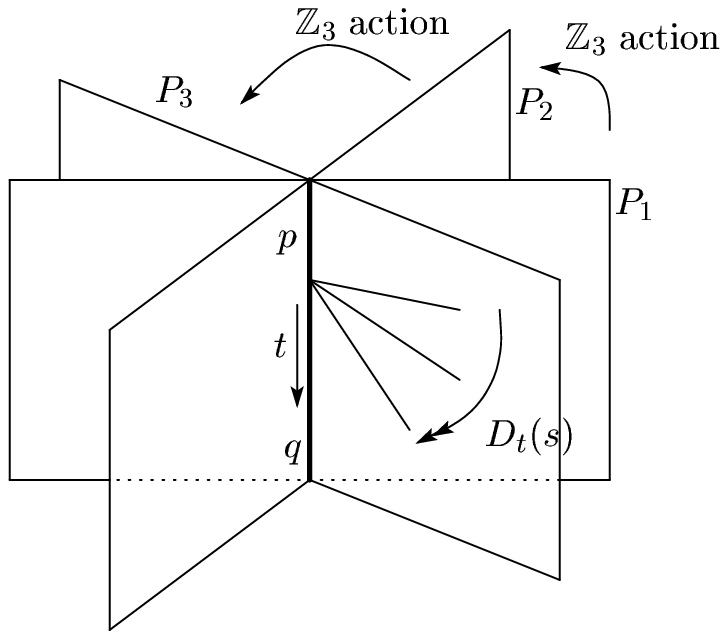}
\end{center}
 
For each $t \in \<p,q\>$, the set of lines passing through $t$ in $P_1$ is parametrized by $\bb P^1$. 
This parameter means direction of the line passing through $t$. 
So we call this parameter $D_t(s)$ where $s\in \bb P^1$. 
For each $t$, 
there is a special direction $D_t(s_0)$ which represent the line $\< p,q\>$. 
Notice that $\< p,q\>$ is in the fixed locus of the $\bb Z_3$ action on $F(Y)$. 
Here recall that $\mca I$ and $\mca J$ defined in the second case of the poof for Lemma \ref{mainlem} 
and let $I_{(a:b)} \in \mca I$ and $J_{(c:d)} \in \mca J$. 
If we consider the blow up $\tilde X_{\bb Z_3}\to X_{\bb Z_3}$, then the tree $\mca I \cup \mca J$ appears as the exceptional curves. 
Now let us consider the limit $\lim_{s\to s_0} D_t(s)$ for $t \in \< p,q\>-\{ p,q \}$. 
This limit is independent of $t$ 
and $\lim_{s\to s_0} D_t(s)$ coincides with the intersection $I \cap \mca J$. 
If we flop $\tilde X_{\bb Z_3} \dashrightarrow W^+$, then $D_p(s)$ and $D_q(s)$ are contracted to points. 
So the fiber $\{(p_i,q_i)\}_{\mu^+}$ is the contraction of Hirtzebruch surface $F_n \to  \tilde F_n$  by $(-n)$-curve. 
In particular, the Picard number of the surface $\{(p_i,q_i)\}_{\mu^+}$ is $1$. 
We can explain these by the following picture. 
\begin{center}
\includegraphics[clip,height=45mm]{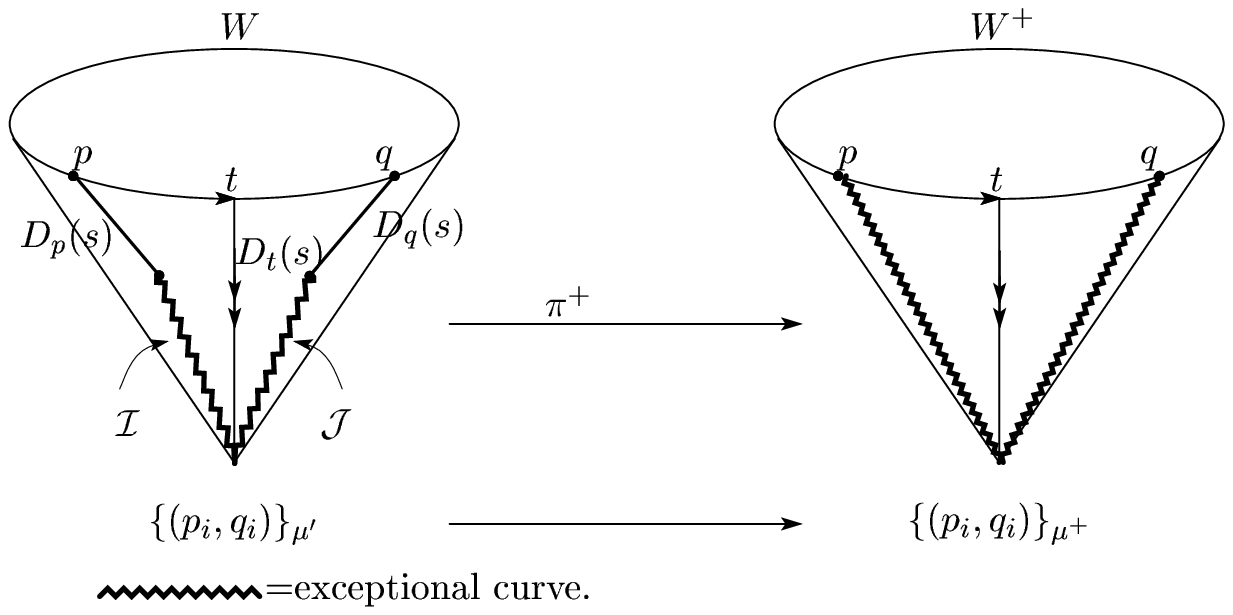}
\end{center}

Let $E$ be the exceptional divisor of $\mu^+$. 
$E$ has a fibration over the singular points of $\bar K^2(C \times D)$. 
A general fiber of the fibration is $\bb P^1$ and there are 81 singular fibers $\tilde F_n$. 
Since the Picard number of $\tilde F_n$ is $1$, 
all curves contracted by $\mu ^+$ are algebraically equivalent to a general fiber 
$$\xi= (\mu^{+})^{-1}( \{ (p,q),(p,q),(-2p,-2q) \}).$$ 
Since the normal bundle of $\xi$ is trivial, the intersection number $(\xi,E)$ is $-2$. 
Therefore $-E$ is $\mu^+$ ample. 
\end{proof}

\subsection{Birational symplectic models of $\tilde X_{\bb Z_3}$}
Finally let us consider other birational symplectic models of $\tilde X_{\bb Z_3}$ which are different from $K^2 (C \times D)$. 
We define divisors needed later. 
Let $E_1$ be the irreducible component of $\mathrm{Exc}(\nu):\tilde X_{\bb Z_3} \to X_{\bb Z_3}$ such that $E_1$ meets $Q_{(\mathrm{I})}$, and $E_1$ does not meet $Q_{(\mathrm{II})}$. 
Similarly, let $E_2$ be the irreducible component of $\mathrm{Exc}(\nu)$ such that $E_2$ meets $Q_{(\mathrm{II})}$, and $E_2$ does not meet $Q_{(\mathrm{I})}$. 
We remark that $\mathrm{Exc}(\nu)$ is $E_1 + E_2$. 
Here let $K^2(C \times D) \dashrightarrow  X_{\mathrm{I}}$ be the Mukai flop on $\bb P_{(\mathrm{I})}$, and 
$K^2(C \times D) \dashrightarrow  X_{\mathrm{II}}$ the Mukai flop on $\bb P_{(\mathrm{II})}$. 

\begin{thm}\label{mainthm2}
Assume that $C$ and $D$ are ``not'' isogenus.
Then there are exactly 4 birational models of $\tilde X_{\bb Z_3}$, that is, 
there are 4 projective irreducible symplectic manifolds $K^2(C \times D)$, $\tilde X_{\bb Z_3}$, $X_{\mathrm{I}}$ and $X_{\mathrm{II}}$ which are birational to $K^2(C \times D)$.  
The movable cone of $\tilde X_{\bb Z_3}$ is decomposed into the ample cones of these 4 models in the following way. 
\begin{center}
\includegraphics[clip,height=37mm]{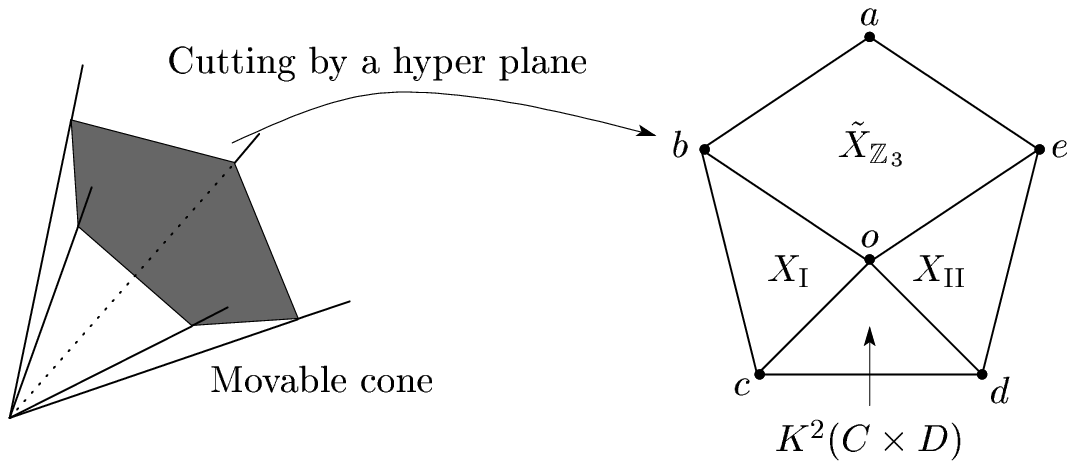}
\end{center}
Each codimension 1 face corresponds to the following contraction morphism. 
\begin{center}
\begin{tabular}{|c|l|cc|}
\hline
Face & Morphism \\ 
\hline
$\overline{oc}$ & Contraction of $\bb P_{(\mathrm{I})}$  \\
\hline
$\overline{od}$ & Contraction of $\bb P_{(\mathrm{II})}$ \\
\hline
$\overline{cd}$ & $K ^2(C \times D) \to \bar K^2(C \times D)$ \\
\hline
$\overline{ab}$ & Contraction of $E_1$ \\
\hline
$\overline{ae}$ & Contraction of $E_2$ \\
\hline
$\overline{ob}$ & Contraction of $Q_{(\mathrm{II})}$ \\
\hline
$\overline{oe}$ & Contraction of $Q_{(\mathrm{I})}$ \\
\hline
$\overline{bc}$ & Contraction of $\tilde E_1$ \\
\hline
$\overline{de}$ & Contraction of $\tilde E_2$ \\
\hline
\end{tabular}
\begin{tabular}{|c|l|}
\hline
Vertex & Morphism \\
\hline
$c$ & $K ^2(C \times D) \to K ^2(C)$ \\
\hline
$d$ & $K ^2(C \times D ) \to K ^2(D)$  \\
\hline
$o$ & Contraction of 18 copies of $\bb P^2$ \\
\hline
$a$ & $\tilde X_{\bb Z_3} \to X_{\bb Z_3}$ \\
\hline
$b$ & contraction of $E_1$ and  $Q_{(\mathrm{II})}$. \\ 
\hline
$e$ & contraction of $E_2$ and $Q_{(\mathrm{I})}$. \\
\hline
\end{tabular}
\end{center}
Here $\tilde E_1$ is the proper transform of $E_1$ in $X_{\mathrm{I}}$, 
$\tilde E_2$ is the proper transform of $E_2$ in $X_{\mathrm{II}}$. 
\end{thm}

We will provide the proof into three parts. We always assume that $C$ and $D$ are not isogenus. 

\begin{lem}\label{Kumcone}
The ample cone of $K^2(C \times D)$ is given by the following picture, 
where $H_1$ (resp.\ $H_2$, $6H_1 + 6H_2 -E$) is the supporting divisor of the vertex $c$ (resp.\ $d$, $o$).  
\begin{center}
\includegraphics[height=3cm,clip]{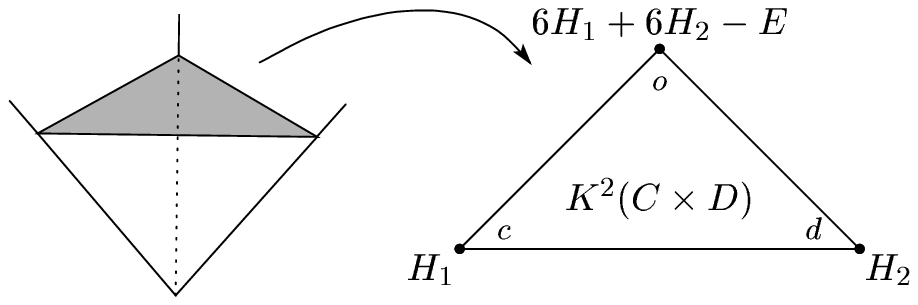}
\end{center}
The divisors $H_1$, $H_2$ and $E$ will be defined in the proof. 
\end{lem}

\begin{proof}
Since $C $ and $D$ are not isogenus, the Picard number of $K^2(C \times D)$ is $3$.
Let $E$ be the exceptional divisor of $\mu: K^2(C \times D) \to \bar K^2(C \times D)$. 
We have the following commutative diagram. 
\[ 
\xymatrix{
   & K ^2(C \times D) \ar[d]^{\mu} \ar[dl]_{\pi_1} \ar[dr]^{\pi_2} & \\
K^2(C)   & \ar[l]_{\bar \pi_1} \bar K^2(C \times D) \ar[r]^{\bar \pi_2} &  K^2(D) 
}
\]
Let $H_C$ (resp. $H_D$) be the pull back of the tautological line bundle $\mca O_{\bb P^2}(1)$ by the isomorphism $K^2(C) \to \bb P^2$ (resp. $K^2(D) \to \bb P^2$). 
We define $H_1$ and $H_2$ by:
\begin{eqnarray*}
H_1:= \pi_1^*(H_C),\  H_2:= \pi_2^*(H_D) .
\end{eqnarray*}
Let $x$ be a line in $\bb P_{(\mathrm{I})}$, $y$ a line in $\bb P_{(\mathrm{II})}$ and $z$ a rational curve which can be contracted by $\mu$. Write $\mu (z) =  \{ (p,q),(p,q), (-2p,-2q)  \}$ and 
assume that neither $p$ nor $q$ are 3-torsion points. 

Since $x$ is contracted by $\pi_1$, the intersection number $(x,H_1)$ is zero. 
Since $\pi_2$ is an isomorphism on one connected component of $\bb P_{(\mathrm{I})}$, we have $(x,H_2)=1$.  
In the same way, $(y,H_2)=0$, and $(y,H_1)=1 $. 
Since $z$ is contracted by $\pi _1$ and $\pi _2$, we have $(z,H_1)=(z,H_2)=0$. 
Since the normal bundle of $z$ in $E$ is trivial and $z$ is isomorphic to $\bb P^1$, $(z,E)=-2$. 
The intersection $E \cap \bb P_{(\mathrm{I})}$ is given by 
\begin{eqnarray*}
E \cap \bb P_{(\mathrm{I})} & = & \bigl\{  \{ (p, q),(p, q) ,(p,-2q) \} \in K^2(C \times D) \big| 3p=0 , q \in D \bigr\} .
\end{eqnarray*}
So, $E \cap \bb P_{(\mathrm{I})}$ is isomorphic to 9 copies of the curve 
$$ \bigl\{  \{ q,q,-2q \} \in K^2(D) | q \in D \bigr\}.$$ 
This curve is isomorphic to the dual curve $D^{\vee }$ of $D$. 
Since $ \deg D^{\vee} =6$, $(x,E) =6$. 
Similarly, we have $(y,E)=6$. 
From these arguments, we have the the following table of intersection numbers. 
\begin{center}
\begin{tabular}{|c|ccc|}
\hline
 & $H_1 $ & $H_2$ & $E$   \\
\hline
$x$ & 0 & 1 & 6 \\
$y$ & 1 & 0 & 6 \\
$z$ & 0 &  0 & $-2$ \\
\hline
\end{tabular}
\end{center}

Clearly both $H_1$ and $H_2$ are nef. 
We will prove that $6 H_1 + 6H_1 -E$ is nef.  
Notice that $6H_1$ (resp. $6H_2$) is lineally equivalent to $\pi _1^* C^{\vee}$ (resp. $\pi _2^* D^{\vee}$). 
Since $\bar \pi _1^* C^{\vee}$ contains the singular locus of $\bar K^2(C \times D)$, 
$\pi ^*_1 C ^{\vee}$ has two irreducible components. 
Let $F_1$  be the proper transform of $\bar \pi_1^* C^{\vee}$ and $F_2$ the proper transform of $\bar \pi_2^* D^{\vee}$. 
Assume that $\pi ^*_1 C ^{\vee} \sim a F_1 + b E$ for some numbers $a$ and $b$. 
Since $(z,F_1)=1$ and $(y,F_1)=0$, we have $\pi_1^* C^{\vee}=2 F_1 +E$. 
Similarly, $ \pi_2^* D^{\vee} \sim 2F_2+E$. 
Assume that $6 H_1 +6 H_2 -E$ is not nef. Then there is a curve $\xi$ such that $( \xi, 6 H_1 +6 H_2 -E) < 0$. 
Since $6H_1 +6H_2 -E$ is linearly equivalent to the following divisors, $\xi$ should be contained in $F_1 \cap F_2$. 
\begin{eqnarray*}
6H_1 +6H_2 -E & \sim & \pi_1^* C^{\vee}  + 6H_2 -E = 2F_1 +6H_2  \\
              & \sim & 6H_1 +  \pi_2^* D^{\vee}  -E = 6H_1 +2F_2 . 
\end{eqnarray*}
By using the cone theorem, we can contract $\xi$. 
Theorem 1 in \cite{Kaw} says that the exceptional locus is covered by rational curves. 
Since $F_1 \cap F_2$ is isomorphic to $C^{\vee} \times D^{\vee}$, 
$F_1 \cap F_2$ does not contain the rational curve. 
So, $6 H_1 +6 H_2 -E$ is nef. 

By the table of intersection numbers, the face $\overline{oc}$ (resp. $\overline{od}$ and $\overline{cd}$) defines the contraction of $x$ (resp. $y$ and $z$). 
\end{proof} 

\begin{lem}\label{F(Y)cone}
The ample cone of $\tilde X_{\bb Z_3} $ is given by$:$ 
\begin{center}
\includegraphics[height=3cm,clip]{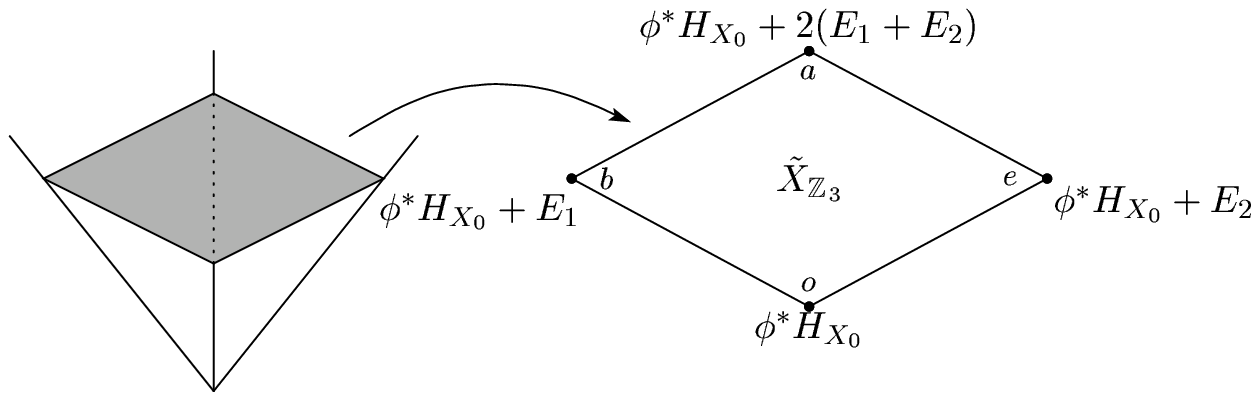}.
\end{center}
Each supporting divisor will be defined in the proof. 
\end{lem}

\begin{proof}
Let $K^2(C \times D) \to  X_0$ be the contraction of $\bb P_{(\mathrm{I})} \cup \bb P_{(\mathrm{II})}$. 
By Theorem \ref{mainthm}, 
there exists a projective morphism $\phi: \tilde X_{\bb Z_3} \to X_0$, and 
we have the following diagram:
\[ 
\xymatrix{
                               &\ar[ld]_{\pi^{+}} W \ar[rd]^{\pi}& \\
K ^2(C \times D)\ar[rd] \ar@{-->}[rr]^(0.53){\mathrm{Mukai flop}}_(0.53){\varphi} &     &\ar[ld]^{\phi} \tilde X_{\bb Z_3} \\
                               & X_0  & \\
}
\]
Let $H_{X_0}$ be an ample divisor of $X_0$. 
Then we may assume that 
$\phi ^* H_{X_0}$ is numerically equivalent to the proper transform of $6H_1 +6H_2-E$. 
Let $x$ be a rational curve in $E_1$ which is contracted by $\nu$, 
$y$ a rational curve in $E_2$ which is contracted by $\nu $, $z$ a line in $Q_{(\mathrm{I})}$, 
and $w$ a line in $ Q_{(\mathrm{II})}$. 

Clearly $(x,E_2)=(y,E_1) = 1 $. 
Since the normal bundle of $x$ in $E_1$ (resp.\ $y$ in $E_2$) is trivial, $(x,E_1)=-2$ (resp.\ $(y,E_2)=-2$). 
By the definition of $X_0$, we have $(z,\phi ^* H_{X_0})=(w,\phi ^*H_{X_0}) =0$. 
Since $ Q_{(\mathrm{I})} \cap E_1$ is isomorphic to $D$, we have $ (z, E_1) = 3$. 
Since $Q_{(\mathrm{I})} \cap E_2 = \emptyset$, we have $(z,E_2)=0$.  
Similarly, we have $(w,E_2)=3$ and $(w,E_1)=0$. 
Recall divisors $F_1$, $F_2$ and $E$ of $K^2(C\times D)$ defined in the proof for Lemma \ref{Kumcone}. 
Let $\tilde{F_1}$, $\tilde{F_2}$ and $\tilde E$ be a proper transform of $F_1$, $F_2 $ and $E$ respectively. 
By Lemma \ref{mainlem}, 
we have $ \tilde{F_1} =  E_1$ and $\ \tilde{F_2} = E_2$.
Here, we will prove that $(x,\phi ^* H_{X_0}) = 2$. 
Let $Z  \in \tilde X_{\bb Z_3}$ be in the curve $x$. 
For each $Z$, the support of $\mca L_Z$ is a line $\< p_0,q_0 \>$ passing through $p_0 \in C$ and $q_0 \in D$. 
Then $x \cap \tilde E$ is identified with 4 points $\{ q \in D | -2q =q_0 \}$. 
Hence $(x, \tilde E)=4$. 
Since $6H_1 +6H_2 -E \sim 2F_1 +2F_2 +E$, we have the following equation: 
\begin{eqnarray*}
4= (x,\tilde E )&=& (x,\phi ^* H_{X_0}- 2 \tilde F_1 - 2 \tilde F_2) \\
                &=& (x, \phi ^* H_{X_0}) -2(x ,E_1) -2(x ,E_2).
\end{eqnarray*}
Hence we have $(x,\phi ^*H_{X_0})=2$. 
In the same way, we have $(y,\phi ^*H_{X_0})=2$. 
Therefore we get the following table of intersection numbers. 
\begin{center}
\begin{tabular}{|c|ccc|}
\hline
   &$\phi ^* H_{X_0}$ &  $E_1$ & $E_2$ \\ \hline
$x$ & $2$ & $-2$ & $1$  \\
$y$ & $2$ & $1$   & $-2$ \\
$z$ & $0$ & $3$   &  $0$ \\
$w$ & $0$ &  $0$  & $3$ \\
\hline
\end{tabular}
\end{center}

We will prove that the divisor $\phi ^*H_{X_0} + 2 E_1 +2E_2$ defines the morphism $\nu :\tilde X_{\bb Z_3} \to X_{\bb Z_3}$. 
Suppose that  the vertex $a$ is generated by $\nu ^* H$, where $H $ is an ample divisor of $X_{\bb Z_3}$.  
Since the Picard number of $\tilde X_{\bb Z_3} $ is 3, 
we put $ \nu ^*H = a  \phi ^*H_{X_0} + b E_1 + c E_2 $ for some numbers $(a,b,c)$. 
Since both $(x,\nu ^*H)=(y, \nu ^*H)=0$, 
we have $b= c =2a$.

We will prove that both $\phi ^* H_{X_0}+E_1$ and $\phi ^* H_{X_0}+E_2$ are nef. 
From symmetry it is enough to prove it only for $\phi ^* H_{X_0}+E_1$.  
Since $\phi ^* H_{X_0} $ is nef, 
the curve $\xi$ such that $(\xi,\phi ^* H_{X_0}+E_1) <0$ should be contained in $E_1$. 
We can contract $\xi$ by the cone theorem, 
and Theorem 1 in \cite{Kaw} says that such a curve must be rational. 
Since $E_1$ is $\bb P^1$ bundle over $C \times D$, it is enough to check that $(x,\phi ^* H_{X_0} +E_1) \geq 0$. 
By the above table of intersection number, we have 
\[
(x,\phi ^* H_{X_0} + E_1) = 0. 
\]
Therefore $\phi ^*H_{X_0} + E_1$ is nef. 

By the table of intersection numbers, we have the correspondence between vertices and contraction morphisms in 
Theorem \ref{mainthm2}. 
\end{proof}

To complete the proof of Theorem \ref{mainthm2}, we consider 
the ample cone of $X_{\mathrm{I}}$ and $X_{\mathrm{II}}$. 
By symmetry, it is enough to prove it for $X_{\mathrm{I}}$. 

\begin{lem}
The ample cones of $X_{\mathrm{I}}$ and $X_{\mathrm{II}}$ are given by$:$ 
\begin{center}
\includegraphics[clip,height=36mm]{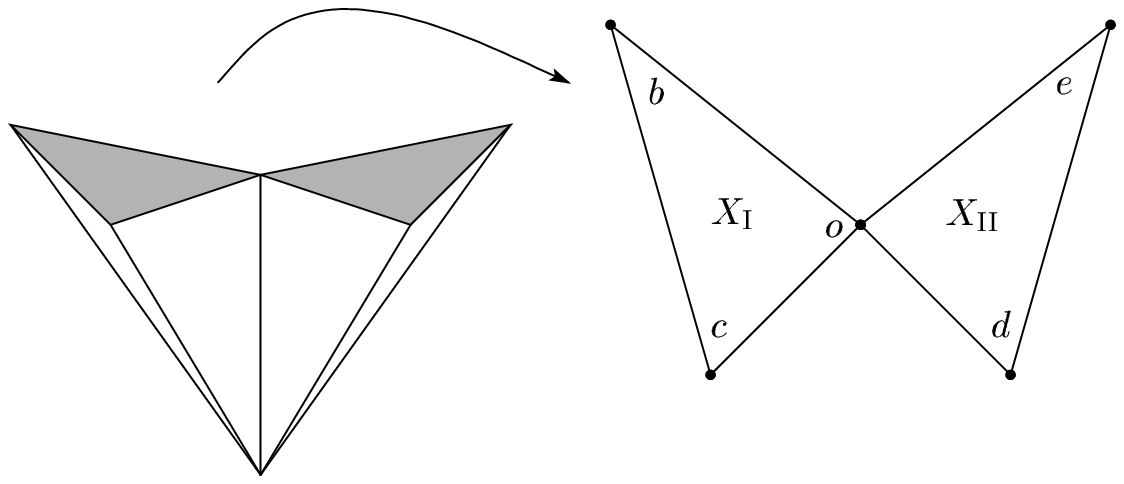}
\end{center}
\end{lem}

\begin{proof}
Throughout the proof,we use the same notations used in the proof of Lemma \ref{F(Y)cone}. 
We only prove that the ample cone of $X_{(\mathrm{I})}$ is a ``tetrahedron''. 
In $\tilde X_{\bb Z_3}$, the vertex $b$ corresponds to the 
contraction of $E_1 \cup Q_{(\mathrm{II})}$. 
We denote by $\tilde{ E_1}  \subset X_{\mathrm{I}}  $ the proper transform of $E_1$ 
and by $\tilde x \subset X_{\mathrm{I}} $ the proper transform of $x$. 
Let $K ^2(C \times D) \to Z_{\mathrm{I}}$ be a contraction of $\bb P_{(\mathrm{I})}$. 
By the construction of $X_{\mathrm{I}}$, 
there exists a birational contraction morphism $\psi_{\mathrm{I}} : X_{\mathrm{I}} \to K^2 (C) $ such that 
the following diagram commutes. 
\[
\xymatrix{
X_{\mathrm{I}} \ar[dr]\ar[ddr]_{\psi_{\mathrm{I}} } &     &\ar@{-->}[ll]_(0.55){{\rm Mukai\ flop\ on}\ \bb P_{(\rm{I})}} K^2(C \times D)\ar[dl]\ar[ddl]^{\pi_1} \\
               & Z_{\mathrm{I}} \ar[d]& \\
             & K ^2(C) &\\
}
\]

In $X_{\mathrm{I}}$, the vertex $c$ corresponds to the contraction morphism $\psi_{\mathrm{I}}$. 
Hence it is enough to prove that $\tilde x$ is contracted by $\psi_{\mathrm{I}}$. 
Assume that $Z \in \tilde X_{\bb Z_3}$ is in the curve $x $. 
For each $Z \in x$, the support of $\mca L_Z$ is the line $\< p,q \>$ passing through $p \in C$ ($3p\neq 0$) and $q \in D$ 
($3q \neq 0$). 
Let $\xi \subset K^2(C \times D)$ be the proper transform of $\tilde x$. 
Then the curve $\xi  \subset K^2(C \times D)$ is given by: 
\[
\xi = \bigl\{ \{  (p,q_1), (p,q_2), (-2p,q) \} \in K^2(C \times D)  | q_1 +q_2 +q=0   \bigr\}.
\]
So $\tilde x$ is contracted by $\psi_{\mathrm{I}}$. 
\end{proof}


\section{Second example}

Let $M$ be the finite subset of $\mathit{PGL}(5)$ defined by
\begin{multline*}
M:=\Bigl\{ 
A:=
\begin{pmatrix}
a_0& & \\
&\ddots & \\ 
&&  a_5\\
\end{pmatrix}
\in \mathit{PGL}(5)  \Big|  
  A\mathrm{\ is \ a\ diagonal\ matrix,\ }  \\
  \#\{ i | a_i=1\} =3, \#\{ j | a_j=\zeta \} =3 \ \cdots \ (*)
\Bigr\}
\end{multline*}
Then $\# M =10$. 
Let $G$ be the finite group generated by the set $M$. 
We define $\tau _i(i=1, \cdots 4) \in G $ respectively  by: \vspace{5pt} 
\begin{eqnarray*}
	\tau _1 &:=& \{ a_0=a_1=a_2=1, a_3=a_4=a_5 =\zeta  \} \\
	\tau _2 &:=& \{ a_0=a_1=a_5=1, a_2=a_3=a_4 = \zeta \} \\
	\tau _3 &:=& \{ a_0= a_1=a_2=a_3= 1, a_4=\zeta ,a_5= \zeta ^2 \} \\
	\tau _4 &:=& \{ a_0=a_3=a_4=a_5 =1, a_1= \zeta , a_2=\zeta ^2\} . 
\end{eqnarray*}
Then $\tau_1, \tau_2, \tau_3$ and $\tau_4$ are generators of $G$. 
So $G$ is isomorphic to $\bb Z_3^{\oplus 4}$. 
Let $Y$ be a smooth cubic 4-fold of Fermat type: 
\[
Y:= \{ (z_0: \cdots : z_5 ) \in \bb P^5 | z_0^3 + \cdots +z_5^3 =0\} .
\]
Let $C$ be the Fermat type elliptic curve:$\{ (x:y:z) \in \bb P^2 | x^3+y^3 +z^3 =0 \}$. 
We will consider the induced $G$-action on $F(Y)$.
\begin{lem}
The $G$-action preserves the symplectic form $\sigma _{F(Y)}$ and $X_G=F(Y)/G$ has a crepant resolution $\tilde X_{G}$.
\end{lem}

\begin{proof}
By Lemma \ref{criterion}, $G$ preserves the symplectic form $\sigma _{F(Y)}$. 
We put $$\alpha _i = \mathrm{diag}(a_{0}^{(i)}, \cdots , a_{5}^{(i)}) \in M \ (i=1, \cdots ,10).$$ 
We may assume that $\alpha _1$ is $\tau_1$ and $\alpha _2$ is $\tau _2$. 
Each $\alpha _i$ generates the cyclic group of order $3$. 
Let $\mathrm{Fix(\alpha _i)}(i=1, \cdots ,10)$ be 
the fixed locus of the action $\langle \alpha _i  \rangle {}^{\curvearrowright }F(Y)$. 
Then $\mathrm{Fix}(\alpha _i) $ is isomorphic to $C \times C$
\begin{center}
\includegraphics[height=50mm,clip]{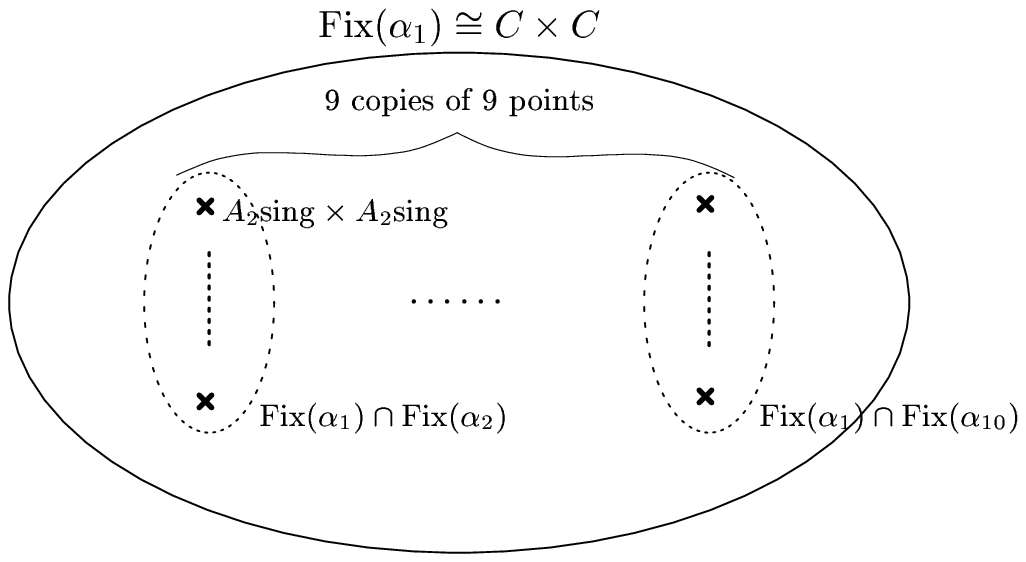}
\end{center}
Each intersection $\mathrm{Fix}(\alpha _i) \cap \mathrm{Fix } (\alpha _j)$($i \neq j$) consists of 9 points. 
Since the $G$-action preserves the symplectic form, singularities of $X_{G}$ is the product of two $A_2$ singularities 
at each point of $\mathrm{Fix}(\alpha _i) \cap \mathrm{Fix } (\alpha _j)$. 
So $X_{G}$ has a crepant resolution $\tilde{X_{G}}$. 
\end{proof}

Since $C$ has a complex multiplication, $\bb Z_ 3$ acts on $C$ in the following way: 
$$
\bb Z_3 \curvearrowright C  ,\  x \mapsto \zeta x .	
$$
We consider the following $\bb Z_3$ action on $C \times C$:
\[
\bb Z_3 {}^\curvearrowright C \times C ,\ (x,y ) \mapsto (\zeta x, \zeta ^2 y).
\]
Let $S$ be the minimal resolution of $C \times C / \bb Z_3$. 
Notice that $S$ is a K3 surface. 

\begin{thm}\label{mainthm3}
Notations being as above, 
$\tilde{X_{G}}$ is birational to $\Hilb ^2 (S)$
\end{thm}

\begin{proof}
To distinguish two elliptic curves $\{(z_0:z_1:z_2) \in \bb P^2 |  z_0^3+ z_1^3+z_2^3=0\}$ and 
$\{ (z_3:z_4:z_5) \in \bb P^2 | z_3^3+z_4^3+z_5^3=0\}$, 
we call them $C$ and $D$ respectively. 
Let $\mf S_3$ be the symmetric group with degree 3. 
By Theorem \ref{mainthm} we have the following birational map: 
\begin{eqnarray*}
\tilde X_{G} \stackrel{\mathrm{birat.}}{\sim} F(Y)/\langle\tau_1, \tau_2,\tau_3 ,\tau_4 \rangle 
&\stackrel{\mathrm{birat.}}{\sim}& \bar K^2(C \times D) / \langle \tau_2,\tau_3 ,\tau_4\rangle \\
&=& \Big( (C \times C \times D \times D )/ \mf S_3 \Big) \Big/\langle\tau_2, \tau_3, \tau_4 \rangle .
\end{eqnarray*}
We will show that $\Big((C \times D)^2 / \mf S_3 \Big) \Big/\langle\tau_2, \tau_3, \tau_4 \rangle$ is birational to 
$\Hilb ^2(S)$. 
Let $x_0$(resp. $y_0$) be a 3-torsion point of $C$(resp. $D$) such that $\zeta x_0 =x_0$. 
The induced actions of $\tau_2, \tau_3$ and $\tau_4$ on $(C \times D)^2$ are give by 
\begin{eqnarray*}
\tau_2 {}^{\curvearrowright} C \times C \times D \times D \ni (x_1,x_2,y_1,y_2) &\mapsto & (\zeta x_1,\zeta x_2,\zeta ^2y_1,\zeta ^2y_2)  \\
\tau_3 {}^{\curvearrowright} C \times C \times D \times D \ni (x_1,x_2,y_1,y_2) &\mapsto & (x_1,x_2,y_1+y_0,y_2+y_0) \\
\tau_4 {}^{\curvearrowright} C \times C \times D \times D \ni (x_1,x_2,y_1,y_2) &\mapsto & (x_1+x_0,x_2+x_0,y_1,y_2).	 
\end{eqnarray*}
Here $\mf S_3$ is generated by  
\[
\mf S_3 = \langle \sigma ,\tau |  \sigma ^3=1, \tau ^2=1, \sigma \tau = \tau \sigma ^2 \rangle ,
\]
and the action $\mf S_3$ on $C \times C \times D \times D$ is given by 
\begin{eqnarray*}
\sigma {}^{\curvearrowright} C \times C \times D \times D \ni (x_1,x_2,y_1,y_2) &\mapsto & (x_2,-x_1-x_2,y_2,-y_1-y_2) \\
\tau    {}^{\curvearrowright} C \times C \times D \times D \ni (x_1,x_2,y_1,y_2) &\mapsto & (x_2,x_1, y_2, y_1). 
\end{eqnarray*}

Let us diagonalize the actions $\sigma {}^{\curvearrowright} C \times C $. 
We have the following diagram:
\[
\begin{CD}
C \times C @>f>> C \times C \\ 
@V\sigma VV @VV\tilde \sigma V \\
C \times C @>f>> C \times C, 
\end{CD}
\]
where the matrix representations of $\sigma ,f $ and $\tilde \sigma $ are respectively 
\begin{eqnarray*} 
\sigma = \begin{pmatrix} 0&1\\ -1 & -1 \end{pmatrix}, \ 
f =\begin{pmatrix} \zeta ^2 & -1 \\ -\zeta &1  \end{pmatrix},\ 
\tilde \sigma = \begin{pmatrix} \zeta  & 0  \\ 0 & \zeta ^2 \end{pmatrix}.
\end{eqnarray*}
For $\tau$, we have the following diagram:
\[
\begin{CD}
C \times C @>f>> C \times C \\ 
@V\tau VV @VV\tilde \tau V \\
C \times C @>f>> C \times C, 
\end{CD}
\]
where the matrix representations of $\tau $ and $\tilde \tau $ are respectively
\[
\tau = \begin{pmatrix} 0&1\\ 1& 0 \end{pmatrix},\ \tilde \tau = \begin{pmatrix} 0 & \zeta \\ \zeta ^2 &0 \end{pmatrix}.
\]

We remark that $(x_1,x_2) \in C \times C$ is in the kernel of $f$ if and only if $\zeta ^2 x_1=x_2, \zeta x_1=x_2$.
So $\Ker f$ is $\{ (0,0), (x_0,x_0), (2x_0,2x_0) \}$. 
In same way for $D$, the action $\sigma {}^{\curvearrowright} D \times D$ can be diagonalized. 
So we have 
\[
\Big( ( C \times C \times D \times D )/ \< \sigma ,\tau  \> \Big) / \< \tau_2, \tau_3,\tau_4 \>
= \Big( (C \times C \times D \times D) / \< \tau_3,\tau_4 \> \Big) / \< \tau_2, \tilde \sigma , \tilde \tau \>.
\]
Hence 
\[
\Big( (C ^2\times D^2)/\mf S_3\Big)\Big/\langle \tau_2,\tau_3,\tau_4 \rangle
  \cong
 (C \times C \times D \times D )/\langle \tilde \sigma, \tilde \tau ,\tau_2\rangle  .
\]
Since  $\langle \tilde \sigma ,\tau_2 \rangle $ is a normal subgroup of $\langle \tilde \sigma, \tilde \tau ,\tau_2\rangle $, we have 
\[
(C \times C \times D \times D )/\langle \tilde \sigma, \tilde \tau ,\tau_2\rangle 
\cong \Big( (C \times C \times D \times D )/\langle \tilde \sigma ,\tau_2 \rangle  \Big) \Big/ \langle \tilde \tau \rangle . 
\]
We can change generators of $\langle \tilde \sigma ,\tau_2 \rangle $ from $\{ \tilde \sigma , \tau_2\}$ to 
$\{ \tau_2 \circ \tilde \sigma , \tau_2 \circ \tilde \sigma ^2  \}$. Each action of $ \tau_2 \circ \tilde \sigma$ and $\tau_2 \circ \tilde \sigma ^2$
are respectively given by 
\begin{eqnarray*}
\tau_2 \circ \tilde \sigma {}^{\curvearrowright } C \times C \times D \times D \ni (x_1,x_2,y_1,y_2) & \mapsto & (\zeta ^2 x_1,  x_2,  y_1, \zeta  y_2) \\
\tau_2 \circ \tilde \sigma {}^{\curvearrowright } C \times C \times D \times D \ni (x_1,x_2,y_1,y_2) & \mapsto & ( x_1, \zeta  x_2, \zeta ^2 y_1, y_2) .
\end{eqnarray*}
We identify $C \times C \times D \times D $ with $(C \times D) \times (C \times D)$ in the following way;
\begin{eqnarray*}
C \times C \times D \times D &\leftrightarrow &  (C \times D) \times (C \times D)\\
 (x_1,x_2,y_1,y_2)                  &\leftrightarrow & \big((x_1,y_2),(x_2,y_1)\big) .
\end{eqnarray*}
The action of $\tilde \tau $ on $(C \times D) \times (C \times D)$ is given by 
\[
\tilde \tau {}^{\curvearrowright} (C \times D) \times (C \times D) \ni (x_1,y_2,x_2,y_1) \mapsto (\zeta ^2x_2 , \zeta y_1, \zeta x_1, \zeta ^2 y_1) .
\]
So we have 
\begin{eqnarray*}
\Big( (C \times C \times D \times D )/\langle \tilde \sigma ,\tau_2 \rangle  \Big) \Big/ \langle \tilde \tau \rangle 
& \cong &  \mathrm{Sym}^2\Big( (C \times D) / \bb Z_3\Big).
\end{eqnarray*}
Here the $\bb Z_3 $ action is the same one as in the first part of this section. So we get the conclusion. 
\end{proof}

\end{document}